\documentclass[preprint,12pt]{elsarticle}

\usepackage{graphicx} 
\usepackage[T1]{fontenc}
\usepackage{comment}
\usepackage{amsmath,amssymb,latexsym,amsfonts,amsthm,comment, xspace, mathrsfs,enumerate,accents}
\usepackage[english]{babel}
\usepackage{algorithmic, algorithm}
\usepackage[normalem]{ulem} 
\usepackage{amsbsy}
\usepackage[latin1]{inputenc}
\usepackage{graphicx}
\usepackage{mathtools}
\usepackage{tikz}
\usepackage{tikz-cd}

\newcommand{\rdp}{\mathbb{R}^d_{+}}
\newcommand{\modu}[1]{\ensuremath{\accentset{\sim}{#1}}}
\newcommand{\ds}{\displaystyle}
 \newcommand{\RR}{\mathbb R}
 \newcommand{\NN}{\mathbb N}
 \newcommand{\CC}{\mathbb C}
 \newcommand{\ZZ}{\mathbb Z}

 \newcommand{\SSS}{\ensuremath{{\cal S}}}

\newtheorem{thm}{Theorem}[section]
 \newtheorem{defn}[thm]{Definition}
 \newtheorem{lem}[thm]{Lemma}
 \newtheorem{prop}[thm]{Proposition}
 \newtheorem{cor}[thm]{Corollary}
 \newtheorem{rem}[thm]{Remark}

\begin{document}

\begin{frontmatter}

\title{Continuity properties of the Laguerre operator, its propagator and related fractional transforms\tnoteref{t1}}
\tnotetext[t1]{2020 Mathematics Subject Classification:
Primary 46F05, 46F12, 44A15; Secondary 35Q41}
\author[1]{Smiljana Jak\v{s}i\'{c}}
\ead{smiljana.jaksic@sfb.bg.ac.rs}
\author[2]{Nenad Teofanov\corref{cor1}}
\ead{tnenad@dmi.uns.ac.rs}
\author[3]{\DJ or\dj e Vu\v{c}kovi\'{c}}
\ead{djordjeplusja@gmail.com}
\cortext[cor1]{Corresponding author}
\affiliation[1]{
  organization={Faculty of Forestry, University of Belgrade},
  addressline={Kneza Vi\v seslava 1},
  city={Belgrade},
  postcode={11000},
  country={Serbia}}

\affiliation[2]{
  organization={Faculty of Sciences, University of Novi Sad},
addressline={Trg Dositeja Obradovi\'{c}a 4},
  city={Novi Sad},
  postcode={21000},
  country={Serbia}
}

\affiliation[3]{
  organization={Technical Faculty ``Mihajlo Pupin'',  University of Novi Sad},
addressline={ \DJ ure \DJ akovi\'{c}a bb},
  city={Zrenjanin},
 postcode={23000},
  country={Serbia}
}

\begin{abstract} We study the well-posedness of a Cauchy problem associated with the general form of the Laguerre operator and relate it to the corresponding global problem for the harmonic oscillator. To this end, we carry out a detailed analysis of the continuity properties of the associated propagator.
  Furthermore, we establish connections between several integral transforms, including the fractional Fourier transform and the fractional Hankel transform. Our results highlight the role of Pilipovi\'c spaces on positive orthants when studying problems involving the Laguerre operator.
\end{abstract}


\end{frontmatter}



\section{Introduction}

The continuity properties of general forms of harmonic oscillator propagators,
are recently studied by Toft, Bhimani, and Manna in  \cite{TBM}. 
They showed that certain classes of time--dependent equations are ill--posed in the framework of classical Gelfand--Shilov spaces, but well--posed in the framework of
suitable global  Pilipovi\'c spaces (i.e. spaces on $\RR^d$).
In this work, we study the Cauchy problem on the positive orthants for the Schr\"odinger type equation involving the
general form of the Laguerre operator. 
We also reveal the connection 
between radial elements from global  Pilipovi\'c spaces and suitable spaces on positive orthants to show that our results 
coincide with the ones from  \cite{TBM} 
when considering the radial solutions.

In addition, we establish connections between several integral transforms, 
including the fractional Fourier transform and the fractional Hankel transform, see  Table \ref{Table_2}  
in Section \ref{sec6}. 
For example, let $E$ denote 
the Laguerre operator:
\begin{equation}
\label{eq:LaguerreOperator}
E=-\sum_{j=1}^d\Big(  \frac{\partial}{\partial x_j}(x_j \frac{\partial}{\partial x_j})-\frac{x_j}{4}+\frac{1}{2} \Big)=-\sum_{j=1}^d \Big(x_j\frac{\partial ^2}{\partial x_j^2}+\frac{\partial}{\partial x_j}-\frac{x_j}{4}+\frac{1}{2}\Big), 
\end{equation} 
$ x_j \in \RR_+$, $ j = 1,\dots, d $, and let  $t\in\RR\setminus\{0\}$. Then we prove that 
$$
  e^{it E}f=\mathcal{I}_{ \mathbf{e^{it}},0}f, \qquad f\in\SSS(\RR^d _{+}),
$$
where $\mathbf{e^{it}}= (e^{it},\dots,e^{it}) \in \CC^d$, and $\mathcal{I}_{\mathbf{e^{it}},0}$  
denotes the fractional power of the Hankel-Clifford transform, see \eqref{eq:frac-pow-H-C}.   
This connection leads to the continuity of the fractional power of the Hankel-Clifford transform, Corollary \ref{cor:hankel}, and its extension to spaces of ultradistibutions, Remark \ref{rem:Hankel-ultradistributions}.
In contrast to \cite{TBM} where the Bargmann transform is used
to prove the relation between the harmonic oscillator  propagator and the fractional Fourier transform, 
here we use a different approach based on kernel representations and suitable change of variables.

To give a flavor of our results, we  proceed
with a particular conclusion included in Theorem \ref{cont. E}.
As a preparation, we introduce necessary notation, 
and refer to Section \ref{sec:Preliminaries} for details.
Let $E$ be given by \eqref{eq:LaguerreOperator} 
and let $\mathcal G_{\alpha}(\RR^d_+)$, $ \alpha >0$, 
be the set of  all $f \in C^{\infty} (\RR^d_+)$ such that
$$
\sup_{x\in\RR^d_+} | E^N f (x)|  \leq C h^{N} N!^\alpha, \qquad N \in \NN,
$$
for some $C, h > 0$. By $\mathcal G' _{\alpha}(\RR^d_+)$,  $ \alpha >0$,  we denote its dual space of ultradistributions, and $\mathcal G_{0}'(\RR^d_+)$ denotes the set of all formal Laguerre series expansions.
Finally, we let the propagator $e^{zE}$, $ z\in \CC$, 
be given by the formal series expansion
\begin{equation}
\label{eE-intro}
e^{zE}f=\sum_{n\in\NN^d_0} e^{z|n|} a_n(f)  l_n,
\qquad f \in \mathcal G_{0}'(\RR^d_+),
\end{equation}
where $l_n $, $n\in\NN^d_0$, denote the Laguerre functions, see \eqref{eq:laguerre}
 and $a_n(f)$, $n\in\mathbb N^d _0$ are the coefficients in the Laguerre series expansion of $f$, see Subsections \ref{subsec:formal}   and \ref{lager-hermit}.
Then we have the following.

\begin{thm}\label{thm:intro}

\begin{itemize}
\item[(i)] Let there be given $ z, c \in\CC $, $z \neq 0$, and $ 0<\alpha < 1$.
The mapping $e^{zE}$ given by (\ref{eE-intro})  restricts to topological isomorphisms on  
$ \mathcal G_{\alpha}(\RR^d_+)$ and $ \mathcal G'_{\alpha}(\RR^d_+).$
\item[(ii)]
If, in addition, $\operatorname{Re} z = 0$, then the mapping (\ref{eE-intro}) 
is  a topological isomorphism on
$ \mathcal G_{\alpha}(\RR^d_+)$ and $ \mathcal G'_{\alpha}(\RR^d_+)$,
for all $\alpha >0$.
 \item[(iii)] If $\alpha>1$ and $\operatorname{Re}(z)<0$ then the mapping (\ref{eE-intro})  is continuous injection but not surjection on 
 $\mathcal G_{\alpha}(\RR^d_+)$ and $\mathcal G'_{\alpha}(\RR^d_+).$
 \item[$(iv)$] If $\alpha>1$ and $\operatorname{Re}(z)>0$ then the mapping (\ref{eE-intro})  is not continuous  on $\mathcal G_{\alpha}(\RR^d_+)$.
 \end{itemize}
 \end{thm}

This result lies in the background of our investigations of well-posedness of the Cauchy problem related to a Schr\"odinger type equations involving the general Laguerre operator of the form $ E_{\rho,c}=\rho E+c=\rho E_{c/\rho},$
$\rho\in\CC\setminus\{0\}$, $c\in \CC$. It turns out that when $ \operatorname{Im}(\rho)>0$, the initial conditions in $ \mathcal G_{\alpha}(\RR^d_+)$ provide well-posedness if and only if $ 0<\alpha < 1$. 

By using the connection between $ \mathcal G_{\alpha}(\RR^d_+)$ and
radially symmetric elements from global Pilipovi\'c spaces (see Subsection \ref{subsec:radially-symmetric}) we show that results  are in accordance with the findings from  \cite{TBM}. We note that 
$ E_{\rho,c} $ can be related to the  two-dimensional harmonic oscillator, see \eqref{fg}.

Moreover, to discuss more general initial conditions, we introduce suitable radial ultradistributions and provide an interpretation of the change of variables in the context of ultradistributions. These results are of independent interest. 

Finally, we establish formulas which show that the Laguerre operator propagator is essentially the same as the fractional Fourier transform, the fractional Hankel transform, and the fractional power of the Hankel-Clifford transform when considering suitable domains.

The  contents of the paper can be briefly summarized as follows. 

We collect the necessary notions related to test function spaces and their distribution spaces in Section \ref{sec:Preliminaries}. Subsection \ref{subsec:radially-symmetric} contains new results concerning radially symmetric spaces, which will be essential for establishing the connection with Schr\"odinger equations associated with the harmonic oscillator. 

Section \ref{sec3} contains a detailed study of the propagator of the Laguerre operator and its continuity properties. In addition to establishing the connection between the propagator and the fractional power of the Hankel--Clifford transform, we prove a more general version of Theorem \ref{thm:intro} and several related results that highlight the difference between $ \mathcal G_{\alpha}(\RR^d_+)$ that are not $G$-type spaces in the sense of 
\cite{Duran92, SSB}, and nontrivial $G$-type spaces.

In  Section \ref{Sec general equations}, we consider the Cauchy problem involving the operator
$E_{\rho,c},  $ $ \rho, c \in\CC $, $\rho \neq 0$,
and discuss continuity results for the corresponding propagators in
Theorems \ref{thm:propagator_T} and \ref{thm:propagators_T and W}. 

In Section \ref{sec5}, we establish the relation between a Cauchy problem involving the Laguerre operator and Schr\"odinger type equations 
considered in \cite{TBM}. More precisely, we provide an explicit relationship between the solutions of the two problems when restricted to the even and radial subspaces of Pilipovi\'c spaces; see Propositions \ref{prop:Toft-i-mi} and 
\ref{rot-inv-novo}. To that end, we introduce and study the auxiliary operator $J_{\rho,c}$, $\rho,c\in\CC$, $\rho\neq 0$, see \eqref{operatorJrc}.
A challenging question concerning the interpretation of these results when the initial values are ultradistributions is addressed in Subsection \ref{subsec:5.2}. The results of Propositions \ref{prop:dist-in-value-1} and \ref{prop:dist-in-value-2} are obtained through careful manipulations involving transposes and inverses of operators, mostly related to the radial symmetry properties. 

In the last section we present formulas that emerge from our investigations and that reveal connections between the propagator of the Laguerre operator and several fractional transforms. These links open up further possibilities for transferring various properties between the operators.

We conclude the paper with Appendix \ref{App-A} containing related results of independent interest.
In particular, we show that the Laguerre operator propagator 
can be expressed as a series when acting on $ \mathcal G_{\alpha}(\RR^d_+)$, $ \alpha \in (0,1)$,
 whereas the same conclusion does not hold for $G$-type spaces.
Finally, we prove a topological isomorphism between suitable spaces of radially symmetric ultradistributions, 
Proposition \ref{dualradial}.

\section{Preliminaries} \label{sec:Preliminaries}

In this section we provide the necessary background for our investigations.
In particular, we introduce the Laguerre operator, spaces of sequences, test functions, and their associated distribution spaces. Subsection \ref{subsec:radially-symmetric} contains several new results on radially symmetric functions and distributions, which will be used in the sequel. 
\subsection{Notation}

The sets of positive
integers, integers, real and complex numbers will be denoted by $\NN $, $\ZZ $, $\RR $ and $\CC$  respectively. Moreover,
$\NN_0=\NN \cup\{0\}$, $\RR_+=(0,\infty)$,  $\overline{\RR_+}=[0,\infty)$.
The upper index $d \in \NN$ will usually denote the dimension, therefore
$\RR^d_+=(0,\infty)^d$, etc.

The standard multi-index notation is used:  for $\alpha\in\NN_0^d$ 
and $x\in\RR^d$ (or $x\in\overline{\RR^d_+}$), $x^\alpha=x_1^{\alpha_1}...x_d^{\alpha_d}$, $0^0=1$,
$|x| = \sqrt{x_1^{2} + \dots + x_d^{2}}$, $x\in\RR^d$,
and $\partial ^\alpha=\partial^{\alpha_1}/{\partial x_1^{\alpha_1}}...\partial^{\alpha_d}/\partial x_d^{\alpha_d}$.
The length of $\alpha \in \NN_0^d$ is given by $\ds |\alpha| = \sum_{i=1}^d \alpha_j$.

The matrix $A=[a_{i,j}]_{1\leq i,j\leq d}$
with real entries 
is orthogonal if it is regular and additionally,
$A^T \cdot A = E_{d\times d}$, where 
$A^T = [a_{j,i}]_{1\leq i,j\leq d}$, and
$E_{d\times d}$ denotes the identity matrix of  order $d$. The space of all orthogonal matrices is denoted by $O(d)$. Since  every $x\in\RR^d$ could be represented as a column-matrix, we have $\sum_{j=1}^d x_j^2=x^T\cdot x$ and therefore, for every $A\in O(d)$ and $y=Ax$ we infer 
\begin{equation}\label{ortogonalnost}
\sum_{j=1}^d y_j^2=(Ax)^T\cdot (Ax)=x^T(A^T\cdot A)x=x^Tx=\sum_{j=1}^d x_j^2.
\end{equation}
In other words, orthogonality preserves the $2-$norm on $\RR^d$. 
%

Recall, a function $f: \RR^d \to \RR$ is {\em even } if
\begin{equation} 
\label{eq:even}
f (x_1, \dots, x_d) = f(\varepsilon_1 x_1, \dots, \varepsilon_d x_d), \qquad
(x_1, \dots, x_d) \in \RR ^d, 
\end{equation}
where  $ \varepsilon_j \in \{ -1, 1\}$, $ j = 1,\dots, d$. 

\par

Next, we introduce notation for the mappings associated with the following changes of variables:
$$
\rho:\RR^d\to \overline{\RR^d_{+}}, \qquad
\rho (x_1,x_2,\dots x_d)=(x_1^2,x_2^2\dots, x_d^2),
$$
$$
\varrho:\overline{\RR^d_{+}}\mapsto \overline{\RR^d_{+}},
\qquad 
\varrho(y_1,y_2,\dots y_d)=(\sqrt{y_1},\sqrt{y_2},\dots,\sqrt{y_d}),
$$
and
\begin{equation}\label{J-intro}
\mathcal{J}(f)= f|_{\overline{\RR^d_{+}}}\circ \varrho,\quad \mathcal J^{-1}(g)= g\circ \rho,
\end{equation}
where $f$ is an even function on $\RR^d$, and
$g$ is an arbitrary function on $\RR^d$.

Let now $f: \RR \to \RR$. By $ \Phi$ we denote the mapping 
$\Phi : f \mapsto \modu{f}$,
given by 
\begin{equation}\label{eq:Phi}
(\Phi  f )(x) =
\modu{f} (x) = f (|x|  ) = f(\sqrt{x_1 ^2 + \dots + x_d ^2}), 
\quad x = (x_1, \dots, x_d) \in \RR^d,
\end{equation}
and 
$\mathcal{J}'$ and $\Phi '$ are the transpose mappings of
 $\mathcal{J}$ and $\Phi $, respectively.

We write $ a_n \preceq  b_n, $ $ n\in\NN$, when 
$\ds \lim_{n\to\infty}a_n=\lim_{n\to\infty} b_n=\infty$ and $\ds \lim_{n\to\infty}\frac{a_n}{b_n}=0$.

Standard notation for spaces of functions, distributions, and sequences will be used throughout the paper.
For example, $ \ell^p $, $ p\geq 1$, denotes the Banach space of { $p-$summable} sequences, endowed with the  usual norm:
$$
\{a_n \}_{n \in \NN ^d _0} \in \ell^p \qquad \Longleftrightarrow \qquad \| a_n \|_{\ell^p}  =
( \sum_{n \in \NN ^d _0} |a_n|^p )^{1/p} <\infty,
$$
and $\displaystyle \{a_n \}_{n \in \NN ^d _0} \in \ell^\infty $ if and only if $
\|a_n \|_{ \ell^\infty } = \sup_{n \in \NN ^d _0} |a_n | < \infty$.

A  measurable function {
$f(x)$, $x \in \RR^d_{+}$, 
} belongs to the Hilbert space  $L^2(\RR^d_+)$ if
$$ \|f\|_{L^2} :=  \ds (\int_{\RR^d_+} |f(x)|^2 dx)^{1/2}<\infty,
$$ and the scalar product is given in the usual way, i.e. 
$ \ds
( f ,g ) = $ $ \int_{\RR^d_+} f(x) \overline{g (x)} dx, $ $ f,g \in L^2(\RR^d_+).$

By $\SSS(\mathbb{R}_+^d)$ we denote the space of all smooth functions, $f\in C^{\infty}(\RR^d_+)$ such that all of its derivatives $\partial^\alpha f$, $\alpha \in\NN^d_0$, extend to continuous functions on
$\overline{\RR^d_+}$, and
$$
\mu_{\alpha, \beta} \; := \sup_{x\in\mathbb{R}^d_+}x^\beta |\partial^\alpha f(x)|<\infty,
\qquad \forall \alpha, \beta \in\mathbb{N}_0^d.
$$
The  system of seminorms $\mu_{\alpha, \beta}$, 
$\alpha, \beta \in\mathbb{N}_0^d $, endows $\mathcal S(\mathbb{R}_+^d)$ with the structure of  Frech\'et  space, see e.g. \cite{Sm}. 
Its strong dual, the space of tempered distributions on orthant, 
is denoted by  $\mathcal{S}'(\mathbb{R}^d_+)  .$
\par

Let $I$ be a directed set and let  $\{X_{\alpha}\}_{\alpha\in I}$ be a family of normed spaces.
 By $ \bigcup_{\alpha\in I} X_{\alpha}$ we denote
 the union of $\{X_{\alpha}\}_{\alpha\in I}$, endowed with  the inductive limit topology, and by
$ \bigcap_{\alpha\in I} X_{\alpha}$ we denote
 the  intersection of $\{X_{\alpha}\}_{\alpha\in I}$, endowed with   
 the projective limit topology.

A Fr\'{e}chet space $X=\bigcap_{n\to\infty} X_n$, where $(X_n)_{n\in\mathbb{N}}$ are Banach spaces, is called a Fr\'{e}chet--Schwartz space ($(FS)$-space) if the inclusions
\[
u_{n}^{\,n+1} : X_{n+1} \to X_n, \qquad n\in\mathbb{N},
\]
are compact operators. On the other hand, if we have  $(X_n)_{n\in\mathbb{N}}$  Banach spaces and compact inclusions
\[
u_{n+1}^{\,n} : X_{n} \to X_{n+1}, \qquad n\in\mathbb{N},
\]
then the space $X=\bigcup_{n\in\NN} X_n$ is called dual Fr\'{e}chet--Schwartz space
($(DFS)$-space).

If $\mathcal{A}$ is a space of test functions, then $\mathcal{A}'$ denotes its dual space of (ultra)distributions, and $_{\mathcal{A}'} \langle\cdot, \cdot \rangle_{\mathcal{A}} =  \langle\cdot, \cdot \rangle$
denotes the dual pairing.

\subsection{Sequence spaces}

For any given  $N\in\ZZ$, 
by $s_N (\NN_0^d)$ 
we denote the space  of all complex sequences $\{a_n\}_{n\in\NN_0^d}$ such that 
\begin{equation}
\label{eq:sequence-s_N}
\|\{a_n \}_{n \in \NN ^d _0}\|_{s_N}:= \|a_n\cdot \langle n\rangle^N\|_{l^{\infty}} <\infty.
\end{equation}
Equipped with this norm $s_N(\NN_0^d)$ becomes a Banach space.

The space  of rapidly decreasing sequences 
 $s(\NN_0^d)$ and 
and its dual space of slowly increasing sequences $s'(\NN^d_0)$
are then defined as follows:
\begin{equation}
\label{eq:sequences-s-and-s'}
s(\NN^d_0)=\bigcap_{N\in\NN_0} s_N(\NN_0^d), \qquad s'(\NN^d_0)=\bigcup_{N\in\NN_0} s_{-N}(\NN_0^d).
\end{equation}
These spaces are $(FN)-$ and $(DFN)-$ spaces, respectively (see \cite[p. 527]{Tr}). 
Every sequence  $\{a_n\}_{n\in\NN_0^d}\in s'(\NN_0^d)$ is identified with the continuous linear functional on $s_N(\NN_0^d)$ in the following sense: it acts on an element $\{b_n\}_{n\in\NN_0^d}\in s_N(\NN_0^d)$ via the mapping 
\begin{equation*}\label{dualniz}
\langle \{a_n\}_{n\in\NN_0^d}, \{b_n\}_{n\in\NN_0^d}\rangle 
:=\sum_{n\in\NN_0^d} a_n\cdot b_n. 
\end{equation*}
It can be shown that, with respect to this identification, $s'(\NN_0^d)$ is the strong dual of $s(\RR^d).$
Moreover,  the spaces $s(\NN_0^d)$ and $s'(\NN_0^d)$ defined in 
\eqref{eq:sequences-s-and-s'}, remain unchanged if  $\ell^{\infty}$-norm in \eqref{eq:sequence-s_N}  is replaced by an arbitrary $\ell^p$-norm, $ p\geq 1$. 

\par 

Let $\alpha>0 $,  $r>0$, and put $\vartheta_{r,\alpha}(n)= e^{r|n|^{1/(2\alpha)}}, $  $ n \in \NN^d_0$.  By $\ell^{p}_{[\vartheta_{r,\alpha}]}(\NN_0^d)$, $p \in [1,\infty]$,
we denote the space of all complex sequences $\{a_n\}_{n\in\NN^d_0}$
such that 
\begin{equation} \label{eq:norm-sequence}
\|\{a_n\}_{n\in\mathbb \NN^d_0 }\|_{\ell^{p}_{[\vartheta_{r,\alpha}]}}:=\|\{|a_n|\cdot \vartheta_{r,\alpha}(n)\}_{n\in\NN^d_0}\|_{\ell^{p}}<\infty.
\end{equation}
Equipped with this norm $\ell^{p}_{[\vartheta_{r,\alpha}]}(\NN_0^d)$ becomes a Banach space.

We  then introduce
\begin{equation}
 \label{eq:sequence_ell}
\ell_{\alpha}(\NN^d_0)=\bigcup_{r>0} \ell^{\infty}_{[\vartheta_{r,\alpha}]}(\NN_0^d), \qquad
  \ell_{0,\alpha}(\NN_0^d) =\bigcap_{r>0} \ell^{\infty}_{[\vartheta_{r,\alpha}]}(\NN_0^d),
\end{equation}
and endow these locally convex spaces with inductive and projective limit topology, respectively.  These spaces are $(DFN)$- and $(FN)$- spaces, respectively.

Their strong duals $\ell_{\alpha}'(\NN_0 ^d)$ and $\ell_{0,\alpha}'(\NN_0 ^d)$ are $(FN)-$ and $(DFN)-$spaces, respectively. Their elements can be identified as sequences so that
\begin{equation}
 \label{eq:sequence_ell_dual}
\ell_{\alpha}'(\NN_0 ^d)=\bigcap_{r>0} \ell^{\infty}_{[1/\vartheta_{r,\alpha}]}(\NN_0^d)
\quad \text{and } \quad
\ell_{0,\alpha}'(\NN_0^d) =\bigcup_{r>0} \ell^{\infty}_{[1/\vartheta_{r,\alpha}]}(\NN_0^d),
\end{equation}
both in algebraic and topological sense.  We refer to
\cite[Proposition 2.2]{SSB} and
\cite[Proposition 1.2]{SSBb} 
for the proof when $\alpha >1$, and the extension to $\alpha > 0$ is straightforward.

We note that the spaces $\ell_{\alpha}(\NN^d_0)$, $\ell_{0,\alpha}(\NN^d_0)$ and their duals remain unchanged if  $\ell^{\infty}_{[\vartheta_{r,\alpha}]}$
and $\ell^{\infty}_{[1/\vartheta_{r,\alpha}]}$ in \eqref{eq:sequence_ell}
and \eqref{eq:sequence_ell_dual}  are replaced by 
$\ell^{p}_{[\vartheta_{r,\alpha}]}$
and $\ell^{p}_{[1/\vartheta_{r,\alpha}]}$ respectively, for any $ p\geq 1.$

Finally, by $\ell_0(\NN^d_0)$ and $\ell'_0(\NN^d)$ we denote the space 
of complex sequences that vanish beyond some index, and   the space of all complex sequences, respectively.

\subsection{  Laguerre and Hermite  functions and  operators} \label{lager-hermit}

For $ n \in\mathbb{N}_0$ and $\gamma\geq 0$, the $n$-th Laguerre
polynomial of order $\gamma$ is defined by
$$
L_n ^\gamma(x)=\frac{x^{-\gamma}e^x}{n!}\frac{d^n}{dx^n}(e^{-x}x^{\gamma+n}),\quad 
x \geq 0,
$$ 
and the $n$-th Laguerre function is defined by
$l_n ^\gamma (x)=L_n ^\gamma (x)e^{-x/2}$, $x \geq 0$.
The $d$-dimensional Laguerre polynomials and Laguerre
functions are given by the corresponding tensor products:
\begin{equation}  \label{eq:laguerre}
L_n^\gamma(x)=\prod_{j=1}^d L_{n_j} ^{\gamma} (x_k) \qquad \text{ and} \qquad
l_n^\gamma(x)=\prod_{j=1}^d l_{n_j} ^{\gamma} (x_k), \quad x \in \overline{\mathbb{R}^d_+},  n\in\NN^d_0.
\end{equation}
We are interested in the case  $\gamma=0$, and write $L_n$ and $l_n$ instead of $L_n^0$ and $ l_n^0$, respectively.

\par 

The Laguerre functions $\{l_n\}_{n\in\NN^d_0}$ form an orthonormal basis in $L^2(\mathbb{R}^d_+)$ (see e.g. \cite{Wong}). Thus every $f\in L^2(\mathbb{R}^d_+)$ can be expressed  via its Laguerre series expansion,
and by $a_n(f)$
we denote the corresponding Laguerre coefficients:
\begin{equation*} \label{eq:Laguerrecoeff}
a_n(f)=\int_{\mathbb{R}^d_+}f(x)l_{n}(x)dx, \qquad n\in\mathbb{N}^d_0.
\end{equation*} 
Therefore,
$$
f=\sum_{n\in\mathbb{N}^d_0}  a_n(f)l_n, \qquad
f\in L^2(\RR^d_+),
$$
where the equality holds in $L^2-$sense. 
\par

Recall, the Laguerre operator $E$ is given by
\begin{equation*}
E=-\sum_{j=1}^d\Big(  \frac{\partial}{\partial x_j}(x_j \frac{\partial}{\partial x_j})-\frac{x_j}{4}+\frac{1}{2} \Big)=-\sum_{j=1}^d \Big(x_j\frac{\partial ^2}{\partial x_j^2}+\frac{\partial}{\partial x_j}-\frac{x_j}{4}+\frac{1}{2}\Big), 
\end{equation*} 
$ x = (x_1, \dots, x_d) \in \RR^d_+.$ The Laguerre functions are the eigenfunctions of the operator $E$:
$$
E ( l_n(x))=\Big(\sum_{i=1}^d  n_j\Big )\cdot l _n(x)=|n| \cdot l _n(x),
$$
see \cite[(11), p.188]{Ed}, and for $N\in \mathbb N$, and $n\in \NN_0^d$ we have
\begin{equation*}\label{Eigen}
E^{N} ( l_n(x))=|n|^N l_{n}(x), \qquad x\in\RR^d_+.
\end{equation*}
If $f, Ef \in L^2(\RR^d_+)$, by using the integration by parts we  get
\begin{equation*} 
\langle Ef, l_n\rangle = \langle f, El_n\rangle=|n|\langle f,l_n\rangle=|n|\cdot a_n(f), \quad n \in \NN_0^d,    
\end{equation*}
i.e. the operator $E$ is a self-adjoint operator on 
$ L^2(\RR^d_+)$.

The Hermite polynomials $H_n$ and the corresponding Hermite functions $h_n$ are defined by
\begin{equation}   \label{eq:hermite-polynomial}
H_n(x) =(-1)^n e^{x^2}\frac{d^n}{dx^n}(e^{-x^2}),
   \qquad  h_n (x)= (2^n n!\sqrt{\pi})^{-1/2} e^{-x^2/2}H_n(x),
\end{equation}
$ x\in\mathbb{R}$, $ n\in\NN_0,$
respectively. The $d$-dimensional Hermite polynomials $H_n$ and the Hermite functions $h_n$ are given by the tensor products
$$ 
H_n(x)= \prod_{j=1}^d  H_{n_j}(x_j),\quad \mbox{and} \quad
h_n(x)=  \prod_{j=1}^d  h_{n_j}(x_j),\quad
x\in\mathbb{R}^d,\; n\in\NN_0^d,
$$
respectively. 

The harmonic oscillator
(the Hermite operator) $H$ is given by:
\begin{equation*}
\label{harmonijski}
H=  -\Delta + |x|^2
=
\sum_{j=1}^d ( - \frac{\partial^2}{\partial x^2_j} +x_j^2),
\qquad
x \in\mathbb{R}^d, 
\end{equation*}
where $\Delta$ denotes the Laplacian in $\RR^d.$
Then the Hermite functions are eigenfunctions of $H$:
$$
H ( h_n(x))=(2|n|+d)h_n (x),  
$$
see \cite[p. 8]{Toft1}, and
$$
H^{N} ( h_n(x))=  (2|n|+d)^N h_{n}(x)
 \qquad
x \in\mathbb{R}^d,  n\in\NN_0^d,  N\in \mathbb N.
$$

The Hermite functions $\{h_n\}_{n\in\NN^d_0}$ form an orthonormal basis in 
$L^2(\mathbb{R}^d)$, thus every $f\in L^2(\mathbb{R}^d)$ 
can be expressed via its Hermite series expansion,
$$
f=\sum_{n\in\mathbb{N}^d_0}  c_n(f)h_n,
$$
where the equality holds in $L^2-$sense, and
\begin{equation}
\label{eq:Hermitecoeff}
 c_n(f)=\int_{\mathbb{R}^d}f(x)h_{n}(x)dx, \qquad n\in\mathbb{N}^d_0.
\end{equation}

\par

The Hermite and Laguerre polynomials are related in the following way:
$$
H_{2n} (x) = (-1)^n 2^{2n} n! L_n ^{-1/2} (x^2), \qquad
H_{2n+1} (x) = (-1)^n 2^{2n} n! L_n ^{1/2} (x^2) x, 
$$
$x\in \mathbb{R} ^d$, $ n\in\NN_0 ^d$,
see (1.1.52) and (1.1.53) in \cite{Thang}.

We end the subsection with an auxiliary result which will be used when studying the connection between radially symmetric and even functions and distributions 
in Subsection \ref{subsec:radially-symmetric}.

\begin{lem} \label{lm:lager-hermit-connection}
Let $l_n$ and $h_n$,  $n\in \NN$, be Laguerre and Hermite functions as given by
\eqref{eq:laguerre} and \eqref{eq:hermite-polynomial}, respectively. Then we have the following connection:
\begin{equation}\label{lagfor2}
l_n(x_1^2+x_2^2)=(-1)^n \sum_{k=0}^n c_{n,k} h_{2k}(x_1) h_{2n-2k}(x_2), \qquad x_1, x_2 \in 
\overline{\RR _+}, n\in \NN_0,
\end{equation}
where the coefficients $c_{n,k}$ satisfy $0\leq c_{n,k}\leq \sqrt{\pi}.$
\end{lem}

\begin{proof}
Let $n\in \NN_0$ be arbitrary but fixed. From \cite[(41) p. 192  and (2) p. 193]{Ed} we obtain
\begin{equation}\label{lager-hermit-polinomi}
L_n(x_1^2+x_2^2)=\sum_{k=0}^n L_k^{-1/2}(x_1^2) L_{n-k}^{-1/2} (x_2^2)=(-1)^n\sum_{k=0}^n \frac{H_{2k}(x_1)H_{2n-2k}(x_2)}{4^n k! (n-k)!}, 
\end{equation}
$ x_1, x_2 \in \overline{\RR _+}$.
Multiplying  \eqref{lager-hermit-polinomi} 
with $\displaystyle e^{-\frac{1}{2} (x_1^2 + x_2^2)}$ 
we obtain the following relation between  Laguerre and Hermite functions: 
\begin{eqnarray*}
 l_n (x_1^2+x_2^2) = (-1)^n \sum_{k=0}^n  \frac{H_{2k}(x_1)e^{-x_1^2/2} H_{2n-2k}(x_2) e^{-x_2^2/2}}{4^n k! (n-k)!} \\
= (-1)^n \sum_{k=0}^n   
 c_{n,k} \cdot h_{2k}(x_1)h_{2n-2k}(x_2), \qquad  x_1, x_2 \in \overline{\RR _+},
\end{eqnarray*}
where
$$
 c_{n,k} = \frac{\sqrt{ \pi   2^{2k} \cdot 2^{2n-2k} (2k)! (2n-2k)!  } }{4^n k! (n- k)!},
 \qquad 0\leq k \leq n.
$$
Thus $c_{n,k}\geq 0$, and the upper bound for $c_{n,k}$ 
can be obtain in the following way:
$$c_{n,k}\leq \frac{\sqrt{\pi}}{2^n}\cdot \sqrt{\frac{(2k)!}{(k!)^2}} \cdot \sqrt{\frac{(2n-2k)!}{\left((n-k)!\right)^2}}\leq \sqrt{\pi}\cdot \frac{\sqrt{2^{2k+2n-2k}}}{2^n}=\sqrt{\pi},$$
 $ n \in \NN _0$, $ 0\leq k \leq n,$
\end{proof}

We note that the $d-$dimensional version of formula \eqref{lagfor2} can be written as
\begin{equation} \label{lagford}
l_n(x_1^2+\cdots +x_d^2)=
\sum_{|k| \leq n} c_{n, k} 
\prod_{j=1} ^d h_{2k_j}(x_j),
 \end{equation}
$ (x_1,\dots x_d)\in\RR^d$, and the coefficients 
$c_{n,k}$ satisfy
$|c_{n,k}|\leq C\cdot n^{2+3+\dots+d-1}$, 
$ k =  (k_1,k_2,\dots, k_d)\in\NN_0^d$,
for some $C>0$ that does not depend on $n$.

\subsection{Test spaces of formal Laguerre and Hermite series} \label{subsec:formal}

By  $\mathcal G_{\alpha}(\RR_{+}^d)$ and $\mathcal G_{0,\alpha}(\RR_{+}^d)$ 
we denote the spaces of formal Laguerre series  expansions
$\sum_{n\in\NN_0^d} a_{n}l_n$ 
that correspond to sequences $\{a_n\}_{n\in\NN^d_0} \in \ell_{\alpha/2}(\NN_0^d)$ and  $\{a_n\}_{n\in\NN^d_0} \in \ell_{0,\alpha/2}(\NN_0^d)$, $\alpha > 0$, respectively (see \eqref{eq:norm-sequence} and \eqref{eq:sequence_ell}). 

\par

Topologies  in $\mathcal G_{\alpha}(\RR_{+}^d)$ and $\mathcal G_{0,\alpha}(\RR_{+}^d)$  are  inherited from the ones in $\ell_{0,\alpha/2}(\NN_0^d)$ and  $\ell_{\alpha/2}(\NN_0^d)$, respectively, via  
the mappings
\begin{eqnarray*} \label{eq:mappingT}
    T: \ell_{\alpha/2}(\NN_0^d) \rightarrow\mathcal G_{\alpha}(\RR_{+}^d), \qquad T(\{a_n\}_{n\in\NN^d_0})=\sum_{n\in\NN_0^d} a_n l_n,  \nonumber
   \\    
   T: \ell_{0,\alpha/2}(\NN_0^d) \rightarrow\mathcal G_{0,\alpha}(\RR_{+}^d), \qquad T(\{a_n\}_{n\in\NN^d_0})=\sum_{n\in\NN_0^d} a_n l_n.
  \end{eqnarray*}
In the same spirit we define the spaces $\mathcal G'_{\alpha}(\RR_{+}^d)$ and $\mathcal G'_{0,\alpha}(\RR_{+}^d)$,
for $\alpha>0$, by using the sequence spaces $\ell_{\alpha/2} '(\NN_0^d)$ and  $\ell_{0,\alpha/2}'(\NN_0^d)$ instead, see \eqref{eq:sequence_ell_dual}. 
 We will also use the notation  $\mathcal G_{0}(\RR^d_{+})$ for the set of finite linear combinations of Laguerre functions and  $\mathcal G'_{0}(\RR^d_{+})$ for the set  of all formal Laguerre series expansions.
 
\par

For the reader's convenience, we  mention Banach spaces which will be 
used in Section 4. These are the images of $\ell^{\infty}_{[\vartheta_{h,\alpha/2}]}(\NN_0^d)$ and $\ell^{\infty}_{[1/\vartheta_{h,\alpha/2}]}(\NN_0^d)$, respectively under the mapping $T$.  More precisely, for a given $h>0$,
 \begin{eqnarray*} \label{eq:G-spaces-alpha-h}
 \mathcal G_{\alpha, h}(\rdp)= \{  \sum_{n\in\NN_0^d} a_n l_n: \|\sum_{n\in\NN_0^d} a_n l_n\|_{\mathcal G_{\alpha, h}} :=\sup_{n\in\NN_0^d} |a_n|e^{h|n|^{1/\alpha}}<\infty \}, \nonumber
 \\
 \mathcal G'_{\alpha,h}(\rdp)=\{  \sum_{n\in\NN_0^d} a_n l_n: \|\sum_{n\in\NN_0^d} a_n l_n\|_{\mathcal G'_{\alpha, h}}:=\sup_{n\in\NN_0^d} |a_n|e^{-h|n|^{1/\alpha}}<\infty\}.
  \end{eqnarray*} 
These spaces are Banach spaces equipped with the norms
$ \| \cdot \|_{\mathcal G_{\alpha, h}} $ and $ \| \cdot \|_{\mathcal G'_{\alpha, h}} $, respectively. Of course, we then have 
\begin{equation}  \label{eq:G-sequence-space}
\mathcal G_{\alpha}(\rdp)=\bigcup_{h>0} \mathcal G_{\alpha,h}(\rdp),\quad \mathcal G'_{\alpha}(\rdp)=\bigcap_{h>0} \mathcal G'_{\alpha, h}(\rdp), \qquad \alpha>0.
\end{equation}


The spaces of formal Hermite series  expansions
$f=\sum_{n\in\NN_0^d} a_{n} h_n$ that correspond to  sequences $\{a_n\}_{n\in\NN^d_0} \in \ell_{\alpha}(\NN_0^d)$, or  $\{a_n\}_{n\in\NN^d_0} \in \ell_{0,\alpha}(\NN_0^d)$, $\alpha \geq 0$, are denoted by $\mathcal H_{\alpha}(\RR ^d)$ and $\mathcal H_{0,\alpha}(\RR ^d)$, respectively, and the natural mappings
\begin{eqnarray*} \label{eq:mappingTHermite}
    T_h: \ell_{\alpha}(\NN_0^d) \rightarrow\mathcal H_{\alpha}(\RR ^d), \qquad T_h (\{a_n\}_{n\in\NN^d_0})=\sum_{n\in\NN_0^d} a_n h_n, \nonumber \\
        T_h: \ell_{0,\alpha}(\NN_0^d) \rightarrow\mathcal H_{0,\alpha}(\RR ^d), \qquad T_h (\{a_n\}_{n\in\NN^d_0})=\sum_{n\in\NN_0^d} a_n h_n,
                   \end{eqnarray*}

\begin{eqnarray*} \label{eq:mappingTHermiteDual}
                   T_h: \ell'_{\alpha}(\NN_0^d) \rightarrow\mathcal H'_{\alpha}(\RR ^d), \qquad T_h (\{a_n\}_{n\in\NN^d_0})=\sum_{n\in\NN_0^d} a_n h_n, \nonumber \\
        T_h: \ell'_{0,\alpha}(\NN_0^d) \rightarrow\mathcal H'_{0,\alpha}(\RR ^d), \qquad T_h (\{a_n\}_{n\in\NN^d_0})=\sum_{n\in\NN_0^d} a_n h_n,                   \end{eqnarray*}
induce  topologies on the target spaces, see \cite{Toft1}.
Utilizing \cite[Proposition 1.3]{JPTV} one can   prove that the spaces $\mathcal H_{\alpha}(\RR ^d)$ and $\mathcal H_{0, \alpha}(\RR ^d)$ are $(DFN)-$ and $(FN)-$spaces, respectively, while their duals are $(FN)-$ space  $\mathcal H'_{\alpha}(\RR ^d)$ and $(DFN)$- space  $\mathcal H'_{0, \alpha}(\RR ^d)$, respectively.

\subsection{Spaces of test functions and their dual spaces }
\label{lager}

In this subsection, we introduce the test function spaces relevant to our investigations. We first recall global spaces, i.e., spaces of test functions defined on $\RR^d$, and then focus our attention on spaces over $\RR^d_+$.

\subsubsection{Global Gelfand-Shilov and Pilipovi\'{c} spaces }

We first recall the definition of (the Fourier transform invariant) classical Gelfand-Shilov spaces.

Let  $\alpha,A>0$. The Banach space of all functions $\varphi\in C^{\infty}(\RR^d)$ satisfying
$$
\Vert\varphi\Vert_{\SSS^{\alpha,A}_{\alpha,A}}: =\sup_{p,k\in\NN^d_0}\frac{\left\|x^k
\partial^p\varphi(x)\right\|_{L^2(\RR^d)}}{
A^{|p+k|} k!^{\alpha} p!^{\alpha}}<\infty,
$$
is denoted by $\SSS^{\alpha,A}_{\alpha,A}(\RR^d)$ 
and $\Vert\cdot \Vert_{\SSS^{\alpha,A}_{\alpha,A}}$ defines a norm on
$\SSS^{\alpha,A}_{\alpha,A}(\RR^d)$.

The Gelfand-Shilov space of Roumieu type 
$\SSS^{\alpha}_{\alpha}(\RR^d)$, $\alpha> 0$,
is given by
$$
\SSS^{\alpha}_{\alpha}(\RR^d)= \bigcup_{A>0}
\SSS^{\alpha,A}_{\alpha,A}(\RR^d).
$$
The space  $\SSS^{\alpha}_{\alpha}(\RR^d)$ is nontrivial if and only if $\alpha\geq 1/2$, cf. \cite[Chapter IV, Section 8.2]{GS2}. Note that  $h_n\in\SSS^{1/2}_{1/2}(\RR^d)$, $n\in\NN^d_0$.
 
The  dual space of $\SSS^{\alpha}_{\alpha}(\RR^d)$, 
$\alpha \geq 1/2$, is the
Gelfand-Shilov space of 
Roumieu type ultradistributions
$(\SSS^{\alpha}_{\alpha})'(\RR^d)$,
endowed  with  the strong dual topology. 

The Gelfand-Shilov space of Beurling type 
$\Sigma^{\alpha}_{\alpha}(\RR^d)$, $\alpha> 0$,
is given by
$$
\Sigma^{\alpha}_{\alpha}(\RR^d)= \bigcap_{A>0}
\SSS^{\alpha,A}_{\alpha,A}(\RR^d).$$
The space  $\Sigma^{\alpha}_{\alpha}(\RR^d)$ is nontrivial if and only if $\alpha > 1/2$, cf. 
 \cite{SP3}. 

The  dual space of $\Sigma^{\alpha}_{\alpha}(\RR^d)$, 
$\alpha > 1/2$, is the
Gelfand-Shilov space of 
Roumieu type ultradistribution
$(\Sigma^{\alpha}_{\alpha})'(\RR^d)$,
endowed  with  the strong dual topology. 

The spaces $\SSS^{\alpha}_{\alpha}(\RR^d)$ 
and $\Sigma ^{\alpha}_{\alpha}(\RR^d)$  obviously increase with respect to $\alpha$,
and  $\Sigma ^{\alpha}_{\alpha}(\RR^d) \subset \SSS^{\alpha}_{\alpha}(\RR^d) $.

\par

Now we introduce Pilipovi\'c spaces on $\RR^d$, \cite{SP3, Toft1} .

The function  $f\in C^{\infty}(\RR^d)$ belong to
the (global) Pilipovi\'{c} spaces of Roumieu type ${\mathbf S}^{\alpha} _{\alpha}(\RR^d)$
(of Beurling type $\mathbf{\Sigma} ^{\alpha} _{\alpha}(\RR^d)$), $ \alpha >0$,
if there exist $h>0$ and $C>0$ (for every $h>0$  there exists $C=C_h>0$)
such that the following estimate holds:
\begin{equation}\label{eq:Pilipovic}
    \|H^N f\|_{L^{\infty}(\RR^d)}\leq C h^N N!^{2\alpha}.
\end{equation}

The following result is Theorem 5.2 in \cite{Toft1}.

\begin{prop}\label{char H and S}
We have 
$${\mathcal H}_{\alpha}(\RR ^d) = \mathbf{S}^{\alpha}_{\alpha}(\RR^d)
\quad \text{and} \quad
\mathcal H_{0,\alpha}(\RR ^d) = \mathbf{\Sigma} ^{\alpha}_{\alpha}(\RR^d), \qquad \alpha > 0.
$$
In addition,
$${\mathcal H}_{\alpha}(\RR ^d) = \SSS^{\alpha}_{\alpha}(\RR^d),
\quad  \alpha \geq 1/2 \quad \text{and} \quad  
\mathcal H_{0,\alpha}(\RR ^d) = \Sigma ^{\alpha}_{\alpha}(\RR^d), \qquad \alpha > 1/2.
$$
\end{prop}

\par

Thus, the spaces of formal Hermite expansions
$\mathcal H_{\alpha}(\RR ^d)$ and $\mathcal H_{0,\alpha}(\RR ^d)$,
coincide with Pilipovi\'{c} spaces $\mathbf{S}^{\alpha} _{\alpha}(\RR^d)$ and  $\mathbf{\Sigma} ^{\alpha} _{\alpha}(\RR^d)$, $\alpha >0$, respectively.
In the sequel we use the notation
$\mathcal H_{\alpha}(\RR ^d)$ and $\mathcal H_{0,\alpha}(\RR ^d)$,
instead of  $\mathbf{S}^{\alpha} _{\alpha}(\RR^d)$ and  $\mathbf{\Sigma} ^{\alpha} _{\alpha}(\RR^d)$,
respectively.

\subsubsection{G-type spaces and $P-$spaces on $\RR^d_{+}$}

Recall that elements from
$ \mathcal{S}(\mathbb{R}^d_+)  $ and $\mathcal{S}'(\mathbb{R}^d_+)  $
admit Laguerre expansions. More precisely, for every $f\in \mathcal{S}(\mathbb{R}^d_+)  $ the sequence $\{a_n(f)\}_{n\in\NN_0^d}  $ belongs to $s(\NN_0^d)$.  Conversely, for any sequence $\{a_n\}_{n\in\NN_0^d}\in s(\NN_0^d)$ the series $\displaystyle\sum_{n\in\NN_0^d} a_n l_n$ converges uniformly and  defines a function in $\mathcal{S}(\mathbb{R}^d_+) $. In the same manner, every sequence $\{a_n\}_{n\in\NN_0^d}\in s'(\NN_0^d)$ gives rise to the element in $\mathcal{S}'(\mathbb{R}^d_+)$ and every element $f$ of the dual $\mathcal{S}'(\mathbb{R}^d_+)$ admits Laguerre  series $f=\sum_{n\in\NN_0^d} a_n(f) l_n$ where $\{a_n (f)\}_{n\in \NN_0^d}\in s'(\NN_0^d)$, \cite[Theorem 3.1]{Sm}.

We also note that 
\begin{equation*} \label{eq:S'_N}
    \mathcal{S}'(\mathbb{R}^d_+)=\bigcup_{N>0} \mathcal{S}'_N (\mathbb{R}^d_+).
\end{equation*}
where
 $$
 \mathcal{S}'_N (\mathbb{R}^d_+)=\{f\in \mathcal{S}'(\mathbb{R}^d_+):
 \|f\|_{\mathcal{S}'_N (\mathbb{R}^d_+)}:=\sup_{n\in\NN_0^d}  a_n(f) \langle n\rangle^{-N}<\infty \}, 
 \qquad N>0.
$$

Next we recall the definition  of $G$-type spaces as well as  their distribution spaces, as given in 
\cite{Duran92, Duran93, SSB, SSBb, SSSB}.

\par

\begin{defn} \label{def:G-type-spaces}
 Let $A>0$. The
space of all $f\in\SSS(\RR^d_+)$ such that
\begin{equation*} \label{seminorme-bez-dodatka}
\sup_{p,k\in\NN^d_0}\frac{\|x^{(p+k)/2} \partial^pf(x)\|_{L^2(\RR^d_+)}}
{A^{|p+k|}k^{(\alpha/2)k}p^{(\alpha/2)p}}<\infty
\end{equation*}
is denoted by $g^{\alpha,A}_{\alpha,A}(\RR^d_+)$. 
\end{defn}

With the seminorms
\begin{equation*} 
\sigma_{A,j}^{\alpha, \alpha} (f)=\sup_{p,k\in\NN^d_0}\frac{\|x^{(p+k)/2} \partial^pf(x)\|_{L^2(\RR^d_+)}}
{A^{|p+k|}k^{(\alpha/2)k}p^{(\alpha/2)p}}+\sup_{\substack{|p|\leq
j\\ |k|\leq j}} \sup_{x\in\RR^d_+}|x^k \partial ^p f(x)|,\,\, j\in\NN_0,
\end{equation*}
it becomes a Frech\'et space.

The spaces $G^{\alpha}_{\alpha}(\RR^d_+)$ and
$g^{\alpha}_{\alpha}(\RR^d_+)$ are 
defined as the inductive and projective limits, respectively, of the spaces
$g^{\alpha,A}_{\alpha,A}(\RR^d_+)$ as $A$ varies:
\begin{equation*}
G^{\alpha}_{\alpha}(\RR^d_+)=
\bigcup_{A>0}
g^{\alpha,A}_{\alpha,A}(\RR^d_+), \qquad
g^{\alpha}_{\alpha}(\RR^d_+)=
\bigcap_{A>0}
g^{\alpha,A}_{\alpha,A}(\RR^d_+).
\end{equation*}
As standard, these spaces are endowed with inductive, respectively projective limit topology.
We refer to  $G^{\alpha}_{\alpha}(\RR^d_+)$, $\alpha\geq 1$, and $g^{\alpha}_{\alpha}(\RR^d_+)$, $\alpha> 1$, as  $G$-type spaces of Roumieu and Beurling type,  respectively.

Their strong duals
are called
$G$-type  ultradistribution spaces 
and are denoted by $(G^{\alpha}_{\alpha})'(\RR^d_+)$ and
$(g^{\alpha}_{\alpha})'(\RR^d_+)$, respectively.

Now we introduce $P-$spaces by using the iterates (powers) of the Laguerre operator, \cite{JPTV}.

\begin{defn} \label{def:PilipovicOrthant}
Let $h>0$, and $\alpha>0$.  We define the Banach space $ \mathbf G^{\alpha,h}_{\alpha,h}(\RR^d_+)$  as the space of all  $ f \in C^{\infty} (\RR^d_+)$
such that
\begin{equation*}\label{eta}
\eta_h^\alpha(f) : =\sup_{N\in\NN}\frac{\Vert E^N f\Vert_{L^2(\RR^d_+)}}{h^{N} N!^\alpha}<\infty.
\end{equation*}

The Pilipovi\'c spaces ($P-$spaces) over $\RR^d_+$, $\mathbf{G}^{\alpha}_{\alpha}(\RR^d_+)$ and  $\mathbf{g}^{\alpha}_{\alpha}(\RR^d_+)$ of Roumieu and Beurling type respectively, are defined by
\begin{equation*}
 \mathbf G^{\alpha}_{\alpha}(\RR^d_+)=
\bigcup_{h>0}  \mathbf
G^{\alpha,h}_{\alpha,h}(\RR^d_+), \qquad \text{and} \qquad
\mathbf{g}^{\alpha}_{\alpha}(\RR^d_+)=
\bigcap_{h>0}  \mathbf
G^{\alpha,h}_{\alpha,h}(\RR^d_+).
\end{equation*}
\end{defn}
The spaces $ \mathbf G^{\alpha}_{\alpha}(\RR^d_+)$ and
$ \mathbf g^{\alpha}_{\alpha}(\RR^d_+)$ are endowed with 
the inductive, respectively projective limit topology.

The strong dual spaces of 
$ \mathbf G^{\alpha}_{\alpha}(\RR^d_+)$ and
$\mathbf{g}^{\alpha}_{\alpha}(\RR^d_+)$ are denoted by 
$ (\mathbf G^{\alpha}_{\alpha})'(\RR^d_+)$ and
$(\mathbf{g}^{\alpha}_{\alpha})'(\RR^d_+)$, respectively.

 We note that the condition $f\in\SSS(\RR^d_+)$ from 
Definition \ref{def:G-type-spaces} is relaxed in Definition \ref{def:PilipovicOrthant}
to $ f \in C^{\infty} (\RR^d_+)$. 
In fact, by \cite[Lemma 3.1]{JPTV}, it follows that
$\Vert E^N f\Vert_{L^2(\RR^d_+)} < \infty $  implies $f\in\SSS(\RR^d_+)$.


We end the section with an important result which
describes $P-$spaces on $\RR^d_{+}$ in terms of suitable corresponding sequence spaces. 

\begin{thm}\label{glavnabeurling} 
\cite[Theorems 3.3, 3.5 and 3.6]{JPTV}
Let $\alpha > 0$.
\begin{itemize}
\item[(i)] If   $f\in \mathcal S(\mathbb R_+^d)$ and $\{a_n(f)\}_{n\in\NN^d_0}\in \ell_{\alpha/2} (\NN_0^d)$   ($\{a_n(f)\}_{n\in\NN^d_0}\in \ell_{0,\alpha/2}(\NN^d_0)$), then 
$
 f \in \mathbf{G}^{\alpha}_{\alpha}(\RR^d_+)$   ($ f \in \mathbf{g}^{\alpha}_{\alpha}(\RR^d_+)$).
\item[(ii)] If  $ f \in \mathbf{G}^{\alpha}_{\alpha}(\RR^d_+)$  ($ f \in \mathbf{g}^{\alpha}_{\alpha}(\RR^d_+)$), 
then $\{a_n(f)\}_{n\in\NN^d_0}\in  \ell_{\alpha/2} (\NN_0^d)$  ($\{a_n(f)\}_{n\in\NN^d_0}\in  \ell_{0,\alpha/2} (\NN_0^d)$).
\item[(iii)] If $a_n\in \ell_{\alpha/2}'(\NN_0 ^d)$ 
($a_n\in \ell_{0,\alpha/2}'(\NN^d _0)$)
then  the series
 $\displaystyle \sum_{n\in\NN^d_{0}}a_n l_n$ converges in the topology of
$(\mathbf{G}^{\alpha}_{\alpha})' (\RR^d_+)$
(in the topology of $(\mathbf{g}^{\alpha}_{\alpha})'(\RR^d_+)$),
 thus defining an element $ f \in (\mathbf{G}^{\alpha}_{\alpha})' (\RR^d_+)$
 ($ f \in (\mathbf{g}^{\alpha}_{\alpha})'(\RR^d_+)$),
 with $ a_n = a_n (f) = (f,l_n)$, $ n \in \NN^d_0$.
\item[(iv)] If $ f \in (\mathbf{G}^{\alpha}_{\alpha})'(\RR^d_+)$
($ f \in (\mathbf{g}^{\alpha}_{\alpha})'(\RR^d_+)$), 
then for the sequence
of its Laguerre coefficients $ \{a_n(f)\}_{n\in\NN^d_0}= \{(f,l_n)\}_{n\in\NN^d_0}$ we have $a_n(f) \in \ell'_{\alpha/2}(\NN^d_{0})$
($a_n(f) \in \ell'_{0,\alpha/2}(\NN^d_{0})$).
\end{itemize}
\end{thm}

Moreover, 
$$
\mathbf{G}^{\alpha}_{\alpha} (\RR^d_+)=\mathcal G_\alpha (\rdp),\quad \mathbf{g}^{\alpha}_{\alpha} (\rdp)=\mathcal G_{0,\alpha}(\rdp),
$$
$$
(\mathbf{G}^{\alpha}_{\alpha})' (\RR^d_+)=\mathcal G'_{\alpha}(\rdp),\quad (\mathbf{g}^{\alpha}_{\alpha})' (\rdp)=\mathcal G'_{0,\alpha}(\rdp), \qquad \alpha > 0,
$$
hold both algebraically and topologically, \cite[Remark 3.4]{JPTV}. 
From these results we conclude 
$$
\mathcal{G}_{\alpha}(\RR^d_+)=G^{\alpha}_{\alpha} (\RR^d_+),\quad \mathcal{G}'_{\alpha}(\RR^d_+)=(G^{\alpha}_{\alpha})' (\RR^d_+),\quad \alpha\geq 1,
$$
$$
\mathcal{G}_{0,\alpha}(\RR^d_+)=g^{\alpha}_{\alpha} (\RR^d_+),\quad \mathcal{G}'_{0,\alpha}(\RR^d_+)=(g^{\alpha}_{\alpha})' (\RR^d_+),\quad \alpha > 1.
$$

For  simplicity,  in the sequel we will use (formal Laguerre series) notation 
$\mathcal G_\alpha (\rdp)$, $\mathcal G_{0,\alpha}(\rdp)$,
$\mathcal G'_\alpha (\rdp)$, and $\mathcal G'_{0,\alpha}(\rdp)$,
for $P-$spaces.

From the previous theorem and \cite[Proposition 1.3]{JPTV} it follows that the spaces $\mathcal{G}_{\alpha}(\RR^d_+)$ and $\mathcal{G}_{0,\alpha}(\RR^d_+)$ are $(DFN)-$ and $(FN)-$spaces respectively, while their duals $\mathcal{G'}_{\alpha}(\RR^d_+)$ and $\mathcal{G'}_{0,\alpha}(\RR^d_+)$ are $(FN)-$ and $(DFN)-$ spaces, respectively.

\par

\subsection{Radially symmetric Pilipovi\'{c} spaces} \label{subsec:radially-symmetric}


A function $f: \RR^d \to \RR$ is called {\em  radially symmetric } 
or {\em radial }, if its value depends only on the distance from the origin. That is, there exists a function $g:[0,\infty) \to \mathbb{R}$ such that
\begin{equation*} 
f(x_1,\dots, x_d) = g\big(\sqrt{x_1^2 + \dots + x_d ^2}\,\big)
= g(r), \qquad (x_1, \dots x_d)\in \mathbb{R}^d,    
\end{equation*}
where $\ds r=\sqrt{x_1^2 + \dots + x_d^2} \geq 0,$
i.e. $f: \RR^d \to \RR$ is radial if there exists $g:[0,\infty) \to \mathbb{R}$ such that $f=\Phi(g)$.

\begin{rem}\label{orthogonal}
Using (\ref{ortogonalnost}) it follows that if  $f:\RR^d\rightarrow \RR$ 
is a  radial function and if $A\in O(d)$, then 
\begin{equation}\label{RadialOrt}
f(Ax)=f(x), \qquad x\in\RR^d.\end{equation}
Moreover, the relation \eqref{RadialOrt} characterizes radial functions. 
\end{rem}

  Obviously, every radial function $f$ is even in the sense of  
\eqref{eq:even}.  By  (\ref{RadialOrt})
it is easy to see that $f:\RR\to\RR$ is even if and only if it is radial, since $O(1)=\{-1,1\}$. 

The space of all even functions from 
$ \mathcal H_{\alpha/2}(\RR ^d) $
($\mathcal H_{0,\alpha/2}(\RR ^d) $), 
$ \alpha > 0$, is denoted by
$\mathcal H_{\alpha/2,even}(\RR^d )$ ( $\mathcal  H_{0,\alpha/2,even}(\RR^d )$).

The space of  radial functions from
$\mathcal  H_{\alpha/2}(\RR^d )$  ($ \mathcal  H_{0,\alpha/2}(\RR^d )$) 
denoted by $\mathcal H_{\alpha/2,radial}(\RR^d )$ ($\mathcal H_{0,\alpha/2,radial}(\RR^d )$) 
is a closed subspace in $\mathcal H_{\alpha/2}(\RR^d )$ ($\mathcal H_{0,\alpha/2}(\RR^d )$).

Every even  function $g:\RR\rightarrow \RR $ gives rise to 
radially symmetric function $\modu{g}$ on $\RR^d$:
\begin{equation}\label{eq:tilda}
\modu{g}(x_1,x_2,\dots, x_d)= (\Phi g) (x_1,x_2,\dots, x_d)
= g(\sqrt{x_1^2+ \cdots + x_d^2}), 
\end{equation}
$ x_1, \dots x_d\in \RR$, see \eqref{eq:Phi}.
Throughout this section, we restrict our attention mainly to the case $d=2$,
since the computations in higher dimensions become considerably more cumbersome (cf. \eqref{lagford}).


%

\begin{lem}\label{radijalna1}
Let  $\alpha>0$ and  $g\in  \mathcal H_{\alpha/2,even}(\RR )$ ($g\in  \mathcal H_{0,\alpha/2,even}(\RR ).$) Then the radially symmetric function $\modu{g} = \Phi g$ given by \eqref{eq:tilda} belongs to 
$ \mathcal H_{\alpha/2, radial}(\RR^2)$  ($ \mathcal H_{0,\alpha/2, radial}(\RR^2)$).
\end{lem}

\begin{proof}
We give the  proof for the Roumieu case, since the proof in the Beurling case is similar. 
We also observe that $ \modu{g} $ is obviously even.

Let
$g\in \mathcal{H}_{\alpha/2,even}(\RR ).$ By \cite[Theorem 4.5]{JPTV} the function $f$ defined by $r\mapsto f(r)=g(\sqrt{r})$, $r\geq 0$, belongs to $ \mathcal G_{\alpha}(\RR_{+}).$ 
From \cite[Theorem 3.3 $(ii)$]{JPTV} it follows that  $f$ admits Laguerre expansion $ \displaystyle
f=\sum_{n=0}^{\infty} a_n(f) \cdot l_n,$ such that
\begin{equation}\label{procena1}
|a_n(f)|\leq C e^{-hn^{1/\alpha}}
\end{equation} 
for some $C,h>0$ (see also Theorem \ref{glavnabeurling} {\em (ii)}). 
Combining Lemma \ref{lm:lager-hermit-connection} with inequality (\ref{procena1}), we conclude  that
$$
\sum_{n=0}^{\infty} a_n(f)(-1)^n l_n (x^2+y^2)=\sum_{n=0}^{\infty} (-1)^n a_n(f)\cdot \left(\sum_{k=0}^n c_{n,k} h_{2k}(x_1) h_{2n-2k}(x_2)\right)
$$
is an absolutely  convergent series. Thus, we can rearrange the order of summation 
to obtain 
$$
\modu{g}(x_1,x_2)=\sum_{m\in \NN_0^2} b_m(\modu{g}) h_{m_1}(x_1)h_{m_2}(x_2), \qquad 
x_1, x_2 \in \RR _+,
$$
where 
 \[
b_m(\modu{g}) =
\begin{cases} 
(-1)^{|m|/2} a_{|m|/2}\cdot  c_{m_1/2,m_2/2}, & \text{if both } m_1 \text{ and } m_2 \text{ are even},\\[2mm]
0, & \text{if } m_1 \text{ or } m_2 \text{ is odd}.
\end{cases}
\]
From the estimate (\ref{procena1}) and Lemma \ref{lm:lager-hermit-connection} 
we get
$$
|b_m(\modu{g})|\leq C \sqrt{\pi} e^{-h \left(\frac{|m|}{2}\right)^{1/\alpha}} 
=
C_1 e^{-h_1 |m|^{1/\alpha}}, \qquad m \in  \NN_0^2,
$$
where $\displaystyle  h_1=\frac{h}{2^{1/\alpha}},$ wherefrom 
$\modu{g}\in \mathcal H_{\alpha/2,even}(\RR^2),$  which proves the Lemma.
\end{proof}

Next we show the opposite direction.

\begin{lem}\label{radial2} 
Let $\alpha>0$ and $\modu{g}\in \mathcal  H_{\alpha/2,radial}(\RR^2)$. Then the function $g$, defined by $g(\pm\sqrt{x_1^2+x_2^2})=\modu{g}(x_1,x_2)$, $(x_1,x_2) \in  \RR^2$, 
belongs to the space 
$ {\mathcal  H}_{\alpha/2,even}(\RR)= {\mathcal  H}_{\alpha/2,radial}(\RR).$ 
\end{lem}

\begin{proof} The function $\modu{g}$ admits Hermite function expansion
$$
\modu{g}(x_1,x_2)=\sum_{(m,n)\in\NN^2_0} b_{m,n}(\modu{g}) h_m(x_1)h_n(x_2), 
\qquad (x_1,x_2) \in  \RR^2,
$$
such that $\displaystyle |b_{m,n}(\modu{g})|\leq C e^{-h(m+n)^{1/\alpha}}$
(see \cite[Theorem 5.2]{Toft1}).
Since $\modu{g}$ is  radial, 
we infer that 
\begin{equation}\label{eq:g_series_expansion}
g(x)=g(-x)=\modu{g}(x,0)=\sum_{(m,n)\in\NN_0^2} b_{m,n}(\modu{g}) h_m(x) h_n(0), 
\qquad x\in\RR.
\end{equation}
Using the fact that  $\modu{g}$ is also an even function, it follows that $b_{m,n}(\modu{g})=0$ 
whenever either $m$ or $n$ is odd. 

Since the series in \eqref{eq:g_series_expansion}
is absolutely convergent, we rearrange its coefficients to obtain 
$$
g(x)=\sum_{\substack{m=0\\ m\ \text{even}}}\Big(\sum_{\substack{n=0\\ n\ \text{even}}}^{\infty} b_{m,n}(\modu{g}) h_n(0)\Big) h_m(x)
= \sum_{\substack{m=0\\ m\ \text{even}}} a_m(g) h_m(x), 
\qquad x \geq 0.
$$
Since $h_m$ is even when $m$ is even, it follows that 
the previous calculation holds for every $x\in\RR$.

The application of  \cite[(15) p. 193]{Ed}, together with the definition of Hermite functions and the estimate $\displaystyle {n\choose k}\leq 2^n$, when $n\in \NN, 0\leq k\leq n$ give
\begin{equation}\label{hermit0}|h_n(0)|=\frac{1}{2^k \left((2k)!\pi^{1/2}\right)^{1/2}}\cdot \frac{(2k)!}{k!}=\sqrt{\frac{(2k)!}{(k!)^2}}\cdot \frac{1}{\pi^{1/4} 2^k}\leq \frac{\sqrt{2^{2k}}}{\pi^{1/4}2^k}\leq 1,\end{equation}
when $n=2k$.

Let $m,n\in \NN_0$ be even. From  \cite[(3.4)]{SSSB} when $\alpha\geq 1$, or \cite[(15)]{JPTV} when $0<\alpha<1$ we yield 
\begin{equation}\label{amn} |b_{m,n}(\modu{g})|\leq C e^{-h_1 m^{1/\alpha}}\cdot e^{-h_1 n^{1/\alpha}}\end{equation} for some $h_1>0$. 
Finally,  by estimates (\ref{hermit0}) and  (\ref{amn}), we obtain
$$|a_m(g)|\leq C\cdot \left( \sum_{n=0}^{\infty} e^{-h_1 n^{1/\alpha}}  \right) e^{-h_1 m^{1/\alpha}}\leq C_1 e^{-h_1 m^{1/\alpha}} $$
which proves the assertion.
\end{proof}

Careful examination of proofs of Lemmas \ref{radijalna1} and \ref{radial2}
reveals that the corresponding mappings are continuous, hence we obtain the following. 

\begin{prop}\label{Phi} The mapping 
$\Phi$ given by \eqref{eq:Phi} is an isometric isomorphism
between ${\mathcal H}_{\alpha/2,even}(\RR) $ and 
${\mathcal H}_{\alpha/2,radial}(\RR^2)$ (between 
${\mathcal H}_{0,\alpha/2,even}(\RR)$ and $ {\mathcal H}_{0,\alpha/2,radial}(\RR^2)$).
\end{prop}

\section{The  Laguerre operator and its propagator} \label{sec3}


In this section we introduce the general form of the Laguerre operator propagator, and then study its integral representation.
This reveals the close connection  between the propagator and 
the fractional powers of the Hankel-Clifford transform, see e.g. \cite{SSB}. Then, in Subsection \ref{subsec:continuity-of-propagator} we study its continuity properties in the context of test function spaces on positive orthants and their dual spaces. 
In particular, we obtain Theorem \ref{thm:intro} 
from the introduction as a special case of Theorem \ref{cont. E}.

\par

Let $E$ be the Laguerre operator given by \eqref{eq:LaguerreOperator}, and
$\ds f=\sum_{n\in\NN_0^d} a_n(f) l_n\in \mathcal G'_{0}(\rdp) $. 
Then we define the general form of  the propagator $e^{zE}$,  $z\in\mathbb C$,  as the operator on the space of all formal series  $\mathcal G'_{\alpha}(\rdp)$ as follows:
\begin{equation}\label{Schrodinger propagator}
e^{zE}f=\sum_{n\in\NN^d_0} e^{z|n|} a_n(f)l_n,\quad \mbox{i.e.} \quad  a_n\left(e^{zE}f\right)=e^{z|n|}\cdot  a_n(f), \ n\in\NN_0^d.
\end{equation}
The injectivity of the mapping follows directly from its definition. 

\par

Let  $f\in L^2(\rdp)$ and $\operatorname{Re}(z)\leq 0$. 
The mapping $e^{zE}:L^2(\rdp)\rightarrow L^2(\rdp)$ is well defined since 
$\{a_n(f)\}_{n\in\NN_0^d}\in \ell ^2(\NN_0^d)$ and  $|e^{zn} a_n|\leq |a_n|$, $n\in\NN_0^d$, so that $e^{zE}f\in L^2(\rdp)$. 

In particular, when $\operatorname{Re}(z) = 0$, the operator  $\displaystyle e^{itE}$,  $t \in \RR$, is unitary and  $(e^{itE})^{-1}= e^{-itE}$ on $L^2(\RR^d_+)$. 
Moreover, it is an isometric isomorphism. 

The operator $E$ can therefore be interpreted as the infinitesimal generator of the unitary one-parameter group $(e^{itE})_{t \in \RR}$.

\subsection{The Laguere operator propagator and the fractional power 
of the Hankel-Clifford transform} \label{subsec3.1}

Let us now suppose that $\operatorname{Re}(z)<0$.
Since  $ e^{zE} l_n=e^{z|n|}l_n$, $n\in\NN_0^d$, and $\sum_{n\in\NN_0^d}\|e^{z|n|}l_n\|_{\ell^2}<\infty $, when $\operatorname{Re}(z)<0$, it follows that
the operator $e^{zE}$ is the Hilbert-Schmidt operator \cite[Proposition 16.8] {MV}, and therefore  admits an integral kernel, \cite[Proposition 16.12]{MV}. 

We now make this statement precise.

\begin{lem}  \label{lema:kernel}
Let $z= \operatorname{Re} (z) +i \operatorname{Im} (z) $, $ \operatorname{Re} (z) <0$. 
Then, $e^{zE}$ is an integral operator on  $L^2(\RR^d_+)$ of the form
\begin{equation*} \label{eq:jezgro}
  e^{zE}f(x)=\int_{\RR^d_+}f(y)K_z(x,y)dy, \qquad x \in \overline{\RR^d_+},
\end{equation*}
where 
$$K_z(x,y)=\sum_{k=0}^\infty e^{zk}\sum_{|n|=k}l_n(x)l_n(y), \qquad 
x,y \in \overline{\RR^d_+}.$$
\end{lem}

\begin{proof}
By \eqref{Schrodinger propagator} it follows that
\begin{eqnarray}\label{expE}
e^{zE}f(x) & = & \sum_{k=0}^\infty e^{zk}\sum_{|n|=k}a_n(f)l_n(x)   \nonumber\\
& = & \sum_{k=0}^\infty e^{zk}\sum_{|n|=k}\int_{\RR^d_+}f(y)l_n(x)l_n(y)dy  \nonumber \\
& = & \int_{\RR^d_+}f(y)
\Big ( \sum_{k=0}^\infty e^{zk}\sum_{|n|=k}l_n(x)l_n(y) \Big )
dy
\nonumber \\
& = & 
\int_{\RR^d_+}f(y)K_z(x,y)dy, \qquad x \in \overline{\RR^d_+},
\end{eqnarray}
and the interchange of summation and integration is justified because the last integral converges absolutely in $L^2(\RR^d_+)$.
\end{proof}

To extend Lemma \ref{lema:kernel} to the integral representation of  $e^{zE}$ when 
$ \operatorname{Re} (z) = 0$, we rewrite the kernel in terms of a modified Bessel function.
We follow the main ideas of the proof of \cite[Lemma 3.2]{Sohani}, 
making suitable modifications where necessary. 
First, observe that
\begin{eqnarray}\label{redovi}
K_z(x,y) & = & \sum_{k=0}^\infty e^{zk}\sum_{|n|=k}\prod_{j=1}^dl_{n_j}(x_j)l_{n_j}(y_j)
\nonumber \\
& = & \prod_{j=1}^d\left(\sum_{n_j=0}^\infty l_{n_j}(x_j)l_{n_j}(y_j)e^{zn_j}\right), 
\end{eqnarray}
$n = (n_1, \dots, n_d) \in \NN_0 ^d,$ 
$x = (x_1, \dots, x_d),$ 
$y = (y_1, \dots, y_d)$ $ \in \overline{\RR^d_+}$,
since the series in \eqref{redovi} is absolutely convergent.

From $ \operatorname{Re} (z)<0$, it follows that $|e^{z}| <1$, 
and we can use the Hille--Hardy formula (see \cite[(20) p. 189]{Ed})
to write the $1-$dimensional counterpart of \eqref{redovi} as follows: 
\begin{equation} \label{eq:bessel}
\sum_{n=0}^\infty l_{n}(x)l_{n}(y)e^{zn}=(1-e^{z})^{-1}e^{-\frac{1}{2}\cdot \frac{(1+e^{z})(x+y)}{1-e^{z}}} I_0\Big(2\frac{(x y e^{z})^{\frac{1}{2}}}{1-e^{z}}\Big),
\quad x,y\in \overline{\RR_+},
\end{equation}
where $I_0$ denotes the modified Bessel function of the first kind and the order zero, i.e.
$$
I_0 (\omega) = \sum_{n=0}^\infty \frac{\left ( \frac{\omega}{2}\right )^{2n}}{(n!)^2}, \qquad \omega \in \CC,
$$
cf. \cite[Chapter VII]{Ed}. Next we apply \eqref{eq:bessel} $d$ times, 
for each respective series on the right hand side of \eqref{redovi}, and conclude that
\begin{equation} \label{Kernel_z}
K_z(x,y)  =  (1-e^{z})^{-d}  \prod_{j=1}^d e^{-\frac{1}{2}\cdot \frac{1+e^{z}}{1-e^{z}}
(x_j+y_j)}
\cdot I_0\Big(2\frac{(x_jy_je^{z})^{\frac{1}{2}}}{1-e^{z}}\Big), 
\end{equation}
$ x = (x_1,  \dots, x_d)$, 
$y = (y_1, \dots, y_d) \in \overline{\RR^d_+}$.

We now modify the result of Lemma \ref{lema:kernel}  for the case
$ \operatorname{Re} (z)= 0$.

\begin{lem} \label{HC}
Let $z = i t $, $ t \in  \RR \setminus \{ 0\} $, and  $f \in \SSS (\RR^d_{+})$.
Then,  $ \displaystyle e^{zE}f(x)=\int_{\RR^d_+}f(y)K_z(x,y)dy $, 
$x \in \overline{\RR^d_+}$, where $K_z$ is given by \eqref{Kernel_z}.
\end{lem}

\begin{proof}
The series on the right-hand side of (\ref{Schrodinger propagator}), namely $\ds \sum_{n\in\NN_0^d} e^{it |n|} a_n(f) l_n $ is absolutely convergent, hence $e^{it E}f$ is well  defined. 

Consider now $z_r=\ln r+it$, $r\in (0,1)$. Then $e^{z_r}=re^{it}$ and Abel's theorem for power series yields:

\begin{equation}\label{pomocni1}
\lim_{r\to 1^{-}} e^{z_r E}f  =\lim_{r\to 1^{-}}  \sum_{n\in\NN_0^d} a_n(f)r^{|n|}e^{it |n|} l_n =\sum_{n\in\NN_0^d} a_n(f)e^{it |n|}  l_n=e^{it E}f. 
\end{equation}

Since $|e^{z_r}|=r<1$,  formula $(\ref{expE})$ holds for every $z_r$, i.e.
\begin{equation}\label{pomocni2}
\ds e^{z_r E} f (x)=\int_{\RR^d_{+}} f(y) K_{z_r} (x,y)dy, \qquad x \in \overline{\RR^d_+}. 
\end{equation}

From \cite[(4.2)]{SSB}  it follows that 
\begin{equation} \label{eq:finite}
\sup_{x,y\in\RR_{+}^d}\prod_{j=1}^d I_0\Big(2\frac{\sqrt{x_jy_je^{it}}}{1-e^{it}}\Big)<\infty. 
\end{equation}
On the other hand, it is easy to verify that $\ds \operatorname{Re}\left( \frac{1+e^{it}}{1-e^{it}}\right)=0$, which, in combination with \eqref{eq:finite} 
proves that  $\ds \int_{\RR^d_{+}} f(y) K_{it} (x,y)dy$ converges absolutely,
where
$K_{it}$ is given by \eqref{Kernel_z}.
Moreover,  
\begin{equation}\label{pomocni3}
\lim_{r\to 1^{-}}\int_{\RR^d_{+}} f(y) K_{z_r}(x,y)dy =\int_{\RR^d_{+}} f(y) K_{it} (x,y)dy.
\end{equation}
Combining (\ref{pomocni1}), (\ref{pomocni2}) and (\ref{pomocni3}) we conclude that 
$$
e^{it E}f=\int_{\RR^d_{+}} f(y) K_{it} (\cdot ,y)dy.$$
\end{proof}

Using Lemma \ref{HC} we can show that the propagator 
$e^{it E}$ is essentially the same as the fractional power of the Hankel-Clifford transform.
Recall, the fractional power of the Hankel-Clifford transform of $f\in\SSS(\RR^d_{+})$ is defined by 
\begin{equation} \label{eq:frac-pow-H-C}
    \mathcal{I}_{\mathbf{z},0}f(x)= (1-z)^{-d} 
\int_{\rdp}f(x) \prod_{j=1}^d e^{-\frac{1}{2}\cdot \frac{1+z}{1-z}(x_j + y_j)}   
I_0\left(2\frac{(x_jy_j z)^{\frac{1}{2}}}{1-z}\right)dy, 
\end{equation}
see \cite[Section 4]{SSB}. 
By Lemma \ref{HC} and definitions of $e^{it E}$  and $\mathcal{I}_{\mathbf{z},0}$,
it follows  that 
\begin{equation} \label{eq:link-za-henkel}
    e^{it E}f=\mathcal{I}_{\mathbf{z},0}f,
\end{equation}
where $z=e^{it},
t\in\RR\setminus\{0\}$, and $\mathbf{z}=(z,\dots,z) \in\CC^d$. 

More generaly, if $\mathbf{z} =(e^{it_1},\dots, e^{it_d})\in\CC^d$, and $ (t_1,t_2,\dots, t_d)\in\mathbb R^d\setminus\{(0,\dots,0)\}$, then we  have
$$
\left(\prod_{j=1}^d e^{it_j E_j}\right)f
= \mathcal{I}_{\mathbf{z},0}f,
$$
where $E_j$ is the Laguerre operator with respect to the variable $x_j$, $ j =1,\dots, d$.
For example, if  $ f \in \SSS(\RR_{+} ^d) \cong\bigotimes\limits_{j=1}^d \SSS(\RR_{+})$, then we get
\begin{equation}\label{tenzorhc}
\mathcal{I}_{\mathbf{z},0}f =\left(\bigotimes_{j=1}^d e^{it_jE_j}\right)f,
\end{equation}

Notice that, when $z=-1$ and $\mathbf{z}= \mathbf{-1}= (-1,\dots,-1) \in\CC^d$  we get
$ e^{i\pi  E} = \mathcal{I}_{\mathbf{-1},0}$,  where 
$$ 
\mathcal{I}_{\mathbf{-1},0}f(x)=\frac{1}{2^d}  \int_{\RR^d_{+}} f(x)\prod_{l=1}^d I_0(\sqrt{x_lt_j})dx  $$
is the $d-$dimensional Hankel-Clifford transform,
see e.g. \cite[p. 9]{SSB}.

\subsection{Continuity properties of the Laguerre operator propagator}
\label{subsec:continuity-of-propagator}

Next we discuss continuity properties of the propagator $e^{zE}$
given by \eqref{Schrodinger propagator}, by studying more general propagators.  
Let  $E_{c}=E+ c$, $c\in\CC$, so that
$$
E_{c}f = \sum_{n\in\NN_0^d}(|n|+c)a_n(f)l_n,
$$
and
\begin{equation}\label{propagator powers of E}
e^{zE_{c}} f = \sum_{n\in\NN_0^d}e^{z(|n|+c)}a_n(f)l_n ,\quad z\in\CC\setminus\{0\},
\qquad 
f\in \mathcal{G}'_0(\RR^d_+).
\end{equation}

\begin{lem} \label{lem:isomorphism_S}
Let there be given $t\in \RR,c\in\CC$. Then the mapping  $e^{itE_c}:\SSS(\mathbb{R}_+^d) \rightarrow \SSS(\mathbb{R}_+^d)$ given by
\eqref{propagator powers of E} is a topological isomorphism. The same is true for  $e^{itE_c}:\SSS'(\mathbb{R}_+^d) \rightarrow \SSS '(\mathbb{R}_+^d)$.
\end{lem}

\begin{proof} 
Let $f\in \SSS(\mathbb{R}_+^d)$. By \cite[Theorem 3.1]{Sm} it follows that 
$\SSS(\mathbb{R}_+^d) $ is isomorphic to  the sequence space
$s(\NN^d_0)$ given by \eqref{eq:sequences-s-and-s'}, wherefrom
$\{a_n(f)\}_{n\in\NN_0^d}\in s(\NN^d_0)$. 
Since $a_n\left(e^{itE_c} f\right)=e^{it(|n|+c)} a_n(f)$, it is obvious that $a_n\left(e^{itE} f\right)\in s(\NN^d_0)$, and therefore  $e^{itE_c}f\in  \SSS(\mathbb{R}_+^d)$ along with the continuity of the mapping. The similar calculation holds for the case of $\SSS'(\mathbb{R}_+^d)$. 
 
Note that $(e^{itE_c})^{-1}=e^{-itE_c}$ on these spaces. 
\end{proof}

Our next result can be viewed as a positive orthants counterpart to \cite[Propositions 0.2 and 0.3]{TBM}.

\par 

\begin{thm} \label{thm:cont-vs-discont}
Let  $e^{zE_c} $,$z\in\CC\setminus \{ 0\}$, $c\in \CC$, be given by
\eqref{propagator powers of E}, and let $\mathcal G_1 (\RR_{+}^d)$ and 
$\mathcal G_1 '(\RR_{+}^d)$ be given by \eqref{eq:G-sequence-space}.
\begin{itemize}
\item[$(i)$] 
If $\operatorname{Re}(z)<0$ then the mapping $e^{zE_c}$ is continuous injection, but not surjection from $\mathcal G_1 '(\RR_{+}^d)$ to $\mathcal G_1 (\RR_{+}^d)$.
\item[$(ii)$]
If $\operatorname{Re}(z) > 0$, then the mapping $e^{zE_c}$  is not continuous   from 
$\mathcal G_1 (\RR_{+}^d)$ to $\mathcal G_1 '(\RR_{+}^d)$.
\end{itemize}
\end{thm}

\begin{proof} {\em (i)} The injectivity of  $e^{zE_c} $ is straightforward. 
Let $z\in\CC \setminus \{ 0\}$ be such that  $\operatorname{Re}(z)<0$, and 
let $f\in \mathcal G_{1}'(\RR^d)$. 
By Theorem \ref{glavnabeurling} {\em (iv)} it follows that for every $h>0$ there exists $C_h>0$ such that $|a_n(f)|\leq C_h e^{h|n|}$. Therefore,
$$
|a_n\left(e^{zE_c}(f)\right)|=e^{\operatorname{Re}(z)|n|+ \operatorname{Re}(zc)}\cdot |a_n(f)|\leq C_h e^{h|n|+\operatorname{Re}(z)|n|+ \operatorname{Re}(zc)}.
$$
A choice of $h \in (0,-\operatorname{Re}(z))$, together with 
Theorem \ref{glavnabeurling} {\em (ii)} shows the continuity of $e^{zE_c} $.

{\em (ii)} Let $f\in  \mathcal G_{1} (\RR^d)$. Then, by Theorem \ref{glavnabeurling} {\em (ii)} it follows that  $|a_n(f)|\leq C e^{h|n|}$ for some $C,h>0,$ and 
$$
|a_n\left(e^{zE_c}(f)\right)|=e^{\operatorname{Re}(z)|n|+ \operatorname{Re}(zc)+h|n|}= e^{\operatorname{Re }(zc)} \cdot  e^{(\operatorname{Re}(z)+h)|n|}.
$$
Therefore, if $z\in\CC \setminus \{ 0\}$ is such that  $\operatorname{Re}(z)>0$, it  follows that  $e^{zE_c}(f)\not\in \mathcal G'_1(\RR^d_{+})$.

It remains to comment on the lack of surjectivity when $\operatorname{Re}(z)<0$. Since   
$$
e^{zE_c}e^{-zE_c}=Id,\qquad \mbox{on} \qquad\mathcal G'_0(\RR^d_+),
$$
the claim follows from the injectivity of $e^{zE_c}$ and  the proof of  part 
$(ii)$ of this theorem. 
\end{proof}

To extended Lemma \ref{lem:isomorphism_S} and Theorem
\ref{thm:cont-vs-discont} to
$P-$spaces on positive orthants, $\mathcal G_{\alpha}(\RR^d_+)$ and $\mathcal G_{0,\alpha}(\RR^d_+)$, $\alpha > 0$, and their dual  spaces, 
we consider the following operator
\begin{equation}\label{E^r}
E^r_c(f)= \sum_{n\in\NN_0^d}(|n| +c)^r a_n(f)l_n, \qquad r\in \RR, c\in\CC\setminus\{0\},
\end{equation}
with $-c\not\in \NN_0 $  when $r<0$. Here, the power  $c^r$,
$c\in\CC\setminus\{0\}$, 
is understood as the principal branch of the power function, and $0^r=0$ when $r>0$.

Moreover, for $z\in\CC\setminus\{0\}$ we also define the propagator  $e^{zE^r_c}$ by 
\begin{equation}\label{eE^r}
e^{z E^r_c}(f)= \sum_{n\in\NN_0^d}e^{z(|n|+c)^r} a_n(f)l_n, 
\qquad
f \in \mathcal{G}'_0(\RR^d_+).
\end{equation}

In the next Theorem we consider continuity properties of $E^r_c$ and $e^{z E^r_c}$.

\begin{thm}\label{cont. E}
(i) Let $ c \in\CC $, $r> 0$, and $\alpha >0$.
The mapping $E^r_c$ on $\mathcal{G}'_0(\RR^d_+)$  given by (\ref{E^r}) 
restricts to a continuous mapping on  
\begin{equation} \label{eq:G-type-spaces}
\mathcal G_{0,\alpha}(\RR^d_+),\quad \mathcal G_{\alpha}(\RR^d_+),\quad \mathcal G'_{\alpha}(\RR^d_+)\quad\mbox{and}\quad\mathcal G'_{0,\alpha}(\RR^d_+).    
\end{equation}
The same is true when $r<0$,  and  $ c \in \CC $ is such that $-c\not\in \NN_0 $.

Thus, the mapping $E^r_c$ is a topological isomorphism on 
the spaces in \eqref{eq:G-type-spaces} when 
$r > 0$ and  $ c \in \CC $, and $ -c\not\in \NN_0 $ when $r < 0$. 

(ii) Let there be given  $ z \in\CC \setminus \{ 0\}$, 
$ c \in \CC$, and $r> 0$. If $0<\alpha_1 \leq 1/r$ and $ 0<\alpha_2 < 1/r$,
then the mapping $e^{z E^r_c}$ on $\mathcal{G}'_0(\RR^d_+)$ 
given by  (\ref{eE^r})
restricts to topological isomorphism on  
$$\mathcal G_{0,\alpha_1}(\RR^d_+),\quad \mathcal G_{\alpha_2}(\RR^d_+),\quad \mathcal G'_{\alpha_2}(\RR^d_+)\quad\mbox{and}\quad\mathcal G'_{0,\alpha_1}(\RR^d_+).$$

If, in addition, $\operatorname{Re} z = 0$, then the mapping  $e^{z E^r_c}$
is  a topological isomorphism on $\SSS(\mathbb{R}_+^d)$, $\SSS'(\mathbb{R}_+^d)$
 and the spaces in \eqref{eq:G-type-spaces}, $\alpha >0$.
\end{thm}

\begin{proof} The proof for {\em (i)}
is an easy  consequence of Theorem \ref{glavnabeurling} and is therefore omitted. 

To prove{\em (ii)}, we assume that $r=1$ for simplicity, since the other cases can be proved in a similar way. 

As in \cite[Proposition 2.1]{TBM} it is sufficient to prove that the mapping 
$$
\{a_n(f)\}_{n\in\NN^d_0}\mapsto\{e^{z(|n|+c)}a_n(f)\}_{\NN^d_0}
$$ 
is continuous on the corresponding spaces  of sequences, and use Theorem \ref{glavnabeurling}. We briefly comment  the case $f\in \mathcal G_{\alpha_2}(\RR^d_+) $,  $ 0<\alpha_2 < 1$, since the other cases can be proved in a similar manner.

Let  $f\in \mathcal G_{\alpha_2}(\RR^d_+) $,  $ 0<\alpha_2 < 1$. Then $|a_n(f)|\leq C e^{-h |n|^{1/\alpha_2}} $ for some $C,h>0$ and, consequently, 
$$\ds \left|e^{z (|n|+c)} a_n(f)\right|= 
C \left|e^{\operatorname{Re}(z) c}\right|  
\left| e^{\operatorname{Re}(z)|n|- h |n|^{1/\alpha_2}} \right|  \leq C_1 e^{-h_1 |n|^{1/\alpha_2}}$$
where $h_1\in (0,h)$  and 
$$
C_1=C \left|e^{\operatorname{Re}(z)c}\right| \sup_{n\in\NN_0} e^{\operatorname{Re}(z)|n|- h |n|^{1/\alpha_2}+h_1 |n|^{1/\alpha_2}}. 
$$ 
Then $ 0<\alpha_2 \leq 1$ implies that $ C_1 < \infty$, wherefrom 
$\ds \{e^{z(|n|+c)}a_n(f)\}_{\NN^d_0}\in \ell _{\alpha_2/2}$ which, by 
Theorem  \ref{glavnabeurling} implies  that 
$\ds e^{z E_c} f\in \mathcal G_{\alpha_2}(\rdp) $,
and the proof is finished.
\end{proof}

\begin{rem}
From the proof of Theorem \ref{cont. E} it follows that, for each $r>0$, 
 the mapping (\ref{eE^r})  is continuous injection (but not surjection) on 
 $$\SSS(\mathbb{R}_+^d), \SSS'(\mathbb{R}_+^d),\mathcal G_{\alpha}(\RR^d_+),\mathcal G'_{\alpha}(\RR^d_+), \qquad \alpha >0,$$
 when  $\operatorname{Re}(z)<0$.  
\end{rem}\label{l2}

Since $e^{itE}:L^2(\RR_{+})\to L^2(\RR_{+})$ is an isometric isomorphism, it follows that  the domain of  the fractional power of Hankel-Clifford transform 
can be extended from $\SSS(\mathbb{R}_+)$ to $L^2(\RR_{+})$. Namely, for $f\in L^2(\RR_+)$ and $z=e^{it}$ we simply define its fractional power of Hankel-Clifford transform via  \eqref{eq:link-za-henkel} i.e. $\mathcal{I}_{\mathbf{z},0}f=e^{it E}f$.


Moreover, by the continuity property of $e^{itE}$ on $L^2(\RR_{+})$ we conclude
that, if $f\in L^2(\RR_{+})$ is the limit of a  sequence $f_n\in\SSS(\RR_{+})$ 
in  $L^2-$sense, then the fractional power of Hankel-Clifford transform 
$\mathcal{I}_{e^{it},0}f$ is the limit of $\mathcal{I}_{e^{it},0}f_n$ in 
$L^2-$sense.

In what follows, we extend the domain of $\mathcal{I}_{e^{it},0}$  further. To do so, we need some technical preliminaries.

Let $\alpha>0$, $\mathcal A\in \{\SSS(\RR_{+}),\mathcal G_{\alpha}(\RR_{+}),\mathcal G_{0,\alpha}(\RR^d_{+})\}$ and $c\in \CC$ and $z\in\CC\setminus\{0\}$ such that $e^{zE_c}(\mathcal A)\subseteq \mathcal A.$ 
Furthermore, choose an arbitrary $f\in \mathcal A'$ , $\varphi\in \mathcal A$. Then we yield  
\begin{equation}\label{dualexp}
_{\mathcal A'}\langle e^{zE_c} f,\varphi\rangle_{\mathcal A}=\sum_{n\in\NN_0} a_n(f) e^{z(n+c)^r}\cdot a_n(\varphi)=_{\mathcal A'}\langle  f,e^{zE_c} \varphi\rangle_{\mathcal A}
\end{equation}

 \par

By Theorem \ref{cont. E} and formula \eqref{eq:link-za-henkel}, 
we derive the following continuity properties for the fractional power of the Hankel-Clifford transform on $P-$spaces.

\begin{cor} \label{cor:hankel}
Let $z\in\CC$ be such that $|z|=1,z\not=1$, and 
$\mathbf{z} = ( z,z,\dots,z) \in \CC^d$. 
Then the fractional power of the Hankel-Clifford transform $\mathcal{I}_{\mathbf{z},0}$ given by
\eqref{eq:frac-pow-H-C} is a topological isomorphism on 
$\mathcal G_{0,\alpha}(\RR^d_+)$ and on $\mathcal G_{\alpha}(\RR^d_+)$,
$  \alpha > 0.$

In particular, the conclusion holds for  for the 
$d-$dimensional Hankel-Clifford transform $\mathcal{I}_{\mathbf{-1},0}$. 

\end{cor}

From \cite[Theorem 2.8]{JPTV} we have
$$
\mathcal G_{\alpha}(\RR_{+}^d)\cong\bigotimes_{j=1}^d \mathcal G_{\alpha}(\RR_{+}),\ \mathcal G_{0, \alpha}(\RR^d_{+})\cong\bigotimes_{j=1}^d \mathcal G_{0, \alpha}(\RR_{+}).
$$ 
This fact, along with  (\ref{tenzorhc}) and \cite[Proposition 43.7]{Tr} is enough to  conclude that the same continuity properties for the  fractional power of the Hankel-Clifford transform $\mathcal{I}_{\mathbf{z},0}$ also hold for a more general case when
$$
\mathbf{z}=(e^{i\theta_1},\dots, e^{i\theta_d})\in\CC^d,\quad (\theta_1,\dots, \theta_d)\in\mathbb R^d\setminus\{(0,\dots,0)\}.
$$

\begin{rem} \label{rem:Hankel-ultradistributions}
By using the standard duality arguments we may extend the domain of the  fractional power of the Hankel-Clifford transform  $\mathcal{I}_{\mathbf{z},0}$
to ultradistributions as follows. For simplicity we briefly comment 
the one-dimensional case.

From Theorem \ref{glavnabeurling} it follows that 
 $f\in \mathcal G'_{\alpha} (\RR_{+} )$  ($f\in \mathcal G'_{0,\alpha} (\RR_{+} )$), $ \alpha > 0$, 
admits the Laguerre functions expansion $f=\sum_{n\in\NN_0} a_n(f) l_n$.
By Theorem \ref{cont. E} and  \eqref{dualexp} we conclude that 
$$
e^{it E}f\in  \mathcal G'_{\alpha}(\RR_{+}) \quad (e^{it E}f\in  \mathcal G'_{\alpha}(\RR_{+} )), \qquad t\in\RR\setminus\{0\},
$$ 
and it acts on  test functions $\varphi=\sum_{n\in\NN_0} a_n(\varphi)l_n\in \mathcal G_{\alpha}(\RR_{+})$ ($\varphi\in \mathcal G_{0,\alpha}(\RR_{+})$)
in the following way:
\begin{eqnarray*}
\langle e^{itE}f,\varphi\rangle 
& = &
\langle f,e^{itE}\varphi\rangle= \langle f,\left(\mathcal{I}_{\mathbf{e^{it}},0}\right) \varphi\rangle
\end{eqnarray*}
alowing us to extend the domain of  $\mathcal{I}_{\mathbf{z},0}$ to
 ultradistributions on  
 $\mathcal G'_{\alpha}(\RR_{+} )$ and $\mathcal G'_{0,\alpha}(\RR_{+} )$. 
 More precisely, when   $f\in \mathcal G'_{\alpha}(\RR_{+} )$ 
 ($f\in \mathcal G'_{0,\alpha}(\RR_{+} )$)
 we have that
 $$
 \mathcal{I}_{\mathbf{e^{it}},0}f=e^{itE}f \in \mathcal G'_{\alpha}(\RR_{+} ) \quad 
 ( \mathcal G'_{0,\alpha}(\RR_{+} ) ), 
 $$
acts on test functions 
$\varphi\in {\mathcal G}_{\alpha}(\RR_{+} )$ 
($\varphi\in {\mathcal G}_{0,\alpha}(\RR_{+} )$) in the usual way:
\begin{equation}\label{hcdistribution}
\langle  (\mathcal{I}_{\mathbf{e^{it}},0} )f,\varphi\rangle
= \langle f,\left(\mathcal{I}_{\mathbf{e^{it}},0}\right)\varphi\rangle.
\end{equation}    
\end{rem}

We conclude this section with 
the following negative result concerning the continuity of the operator 
$e^{z E^r_c}$.

\begin{thm} \label{discon E}
Let the operator $e^{z E^r_c}$ be given by (\ref{eE^r}), where 
$z\in\CC$ is such that $\operatorname{Re}(z)>0$, $c\in\CC$, $r>0$, 
$\alpha_1 > \frac{1}{r}$ and $\alpha_2 \geq \frac{1}{r} $. Then the mapping $e^{z E^r_c}$ is discontinuous from
\begin{itemize}
\item[(i)] $\SSS(\mathbb{R}_+^d)$ to $\SSS'(\mathbb{R}_+^d)$;
\item[(ii)] 
$\mathcal G_{0,\alpha_1}(\RR^d_+)$ to $\mathcal G'_{0,\alpha_1}(\RR^d_+)$;
\item[(iii)] $\mathcal G_{\alpha_2}(\RR^d_+)$ to $\mathcal G'_{\alpha_2}(\RR^d_+)$.
\end{itemize}
\end{thm}

\begin{proof}
The proof is analogous to the proof of \cite[Proposition 2.2]{TBM}. We discuss only the case $r=1$, since the general case is entirely analogous. 


We first prove part $(i)$. Let 
$$a_n(f)=e^{-(\log(1+|n|))^2}, \qquad n\in\NN^d_0,
$$ 
be the Laguerre coefficients of $f$. Then
$$f=\sum_{n\in\NN^d_0}e^{-(\log(1+|n|))^2} l_n\in \SSS(\mathbb{R}_+^d)$$
since for every $N\geq 1$ there exists $C_N>0$ such that 
$$e^{-(\log(1+|n|))^2} \leq C_N  \langle n \rangle ^{-N}, \qquad n\in\NN^d_0.$$
On the other hand, it follows from (\ref{propagator powers of E})
$$\left|a_n(e^{zE_{c}}f)\right|=e^{\operatorname{Re}(z)(|n|+c)}e^{-(\log(1+|n|))^2} \succeq e^{\operatorname{Re} (z)\cdot (|n|+c)} \succeq \langle n \rangle ^{N},$$
for every $N\geq 1$  and $n\in\NN^d_0$. Hence, $e^{zE_{c}}f\notin \SSS'(\mathbb{R}_+^d)$.

Next, we prove $(ii)$ and leave $(iii)$ to the reader. We assume that $\alpha\in\RR$ is such that $1<\alpha<\alpha_1$. Let $a_n(f)=e^{-(1+|n|)^{1/\alpha}}$, $n\in\NN^d_0$, be the sequence of Laguerre coefficients of $f$. Then
$$f=\sum_{n\in\NN^d_0} e^{-(1+|n|)^{1/\alpha}} l_n\in \mathcal G_{0,\alpha_1}(\RR^d_+)$$
since
$$e^{-(1+|n|)^{1/\alpha}} \preceq e^{-h(1+|n|)^{1/\alpha_1}},\qquad n\in\NN^d_0,$$
for every $h\geq 1$.

On the other hand,  from (\ref{propagator powers of E}) it follows that
\begin{eqnarray*}
\left|a_n(e^{zE_{c}}f)\right|
& = & e^{\operatorname{Re}(z)(|n|+c)}e^{-(1+|n|)^{1/\alpha}} 
\succeq e^{ \operatorname{Re}(z)\cdot (|n|+c)} \\
& \succeq & e^{h(1+|n|)^{1/\alpha}}  
\succeq  e^{h(1+|n|)^{1/\alpha_1}}, \qquad n\in\NN^d_0,
\end{eqnarray*}
for every $h\geq 1$. Hence, $e^{zE_{c}}f\notin \mathcal G'_{0,\alpha_1}(\RR^d_+)$ and the proof is finished.
\end{proof}

Let $z\in\CC$ be such that $\mbox{Re}(z)>0$, and let $c\in\CC$.
From Theorem \ref{discon E} it follows that 
the propagator $\ds e^{zE _c}$
is discontinuous on all nontrivial $G-$type spaces
$\mathcal G_{0,\alpha_1}(\RR^d_+)$ and 
$\mathcal G_{\alpha_2}(\RR^d_+)$, i.e. when the indices are above
the critical values,  $\alpha_1 >  1$ and $\alpha_2 \geq 1$.
On the other hand, by Theorem \ref{cont. E} {\em (ii)} it is 
continuous on all $P-$type spaces
when the corresponding $G-$type spaces are trivial.

\begin{rem} The proof of Theorem \ref{discon E} shows that the mapping 
$e^{z E^r_c}$ is not only discontinuous but, in fact, not even well defined on the spaces appearing in its formulation.
\end{rem}

\section{A time-dependent equation with the Laguerre operator}
\label{Sec general equations}

In this section we consider a Cauchy problem involving  
the Laguerre operator, and derive continuity properties for the corresponding operator propagators. 
These results are thereafter used in Section \ref{sec5}.

We first introduce the general form $E_{\rho,c}$ of the Laguerre operator $E$ defined by \eqref{eq:LaguerreOperator}:
\begin{equation} \label{eq:E_ro_ce}
E_{\rho,c}=\rho E+c=\rho E_{c/\rho}, \qquad 
\rho\in\CC\setminus\{0\}, \quad  c \in \CC.
\end{equation}
By \eqref{eE^r} and \eqref{eq:E_ro_ce} it follows that the related propagator 
$e^{-itE^r_{\rho,c}}f$ 
is given by
\begin{equation} \label{eq:general_propagator}
e^{-itE^r_{\rho,c}}f = \sum_{n\in\NN_0^d}e^{-it(\rho |n|+c)^r} a_n(f)l_n, 
\qquad
f=\sum_{n\in\NN_0^d} a_n(f) l_n  \in \mathcal{G}'_0(\RR^d_+).    
\end{equation}

We consider the Cauchy problem 
\begin{equation}\label{E Schrodinger problem}
 \begin{array}{ll}
         & i\partial_t u(t,x)-E^r_{\rho, c}u(t,x)=F(t,x),\qquad 
         (t,x)\in [0,T] \times \RR^d_+,\\
         & u_0 (x) = u(0,x)=f(x), \end{array} 
\end{equation}  
with sutiable $F$ and $f$, and the  operator $E^r_{\rho,c}$,  $r>0$,
acting on  the $x-$variable. Note that \eqref{E Schrodinger problem} is a PDE when $r=1$, and pseudodifferntial equation otherwise, cf. \cite{TBM}. 

The  formal solution $u$ of the initial value problem
 \eqref{E Schrodinger problem} is  given by 
\begin{equation}\label{E Schrodinger resenje}
u(t,x)=(e^{-itE^r_{\rho,c}}f)(x) - i\int_{0}^t(e^{-i(t-s)E^r_{\rho, c}}F(s, \cdot))(x)ds, 
\end{equation}
$ (t,x)\in [0,T] \times \RR^d_+,$ and we assume that
$ u(t,x)\in {\mathcal G'_0(\RR^d)}^{[0,T]}$, i.e. for every $t\in [0,T]$ we have $u(t,\cdot)\in \mathcal {G}'_0(\RR^d).$

We recall that the problem (\ref{E Schrodinger problem}) is considered to be 
well-posed if the solution $u(t,x)$ continuously depends 
on the initial value $u_0$.
Hence, it is crucial to examine the continuity properties of the propagators 
\begin{equation}\label{eq:operator-T}
(T_{r,\rho,c} f)(t,x)=\left(e^{-itE^r_{\rho,c}}f\right)(x), \quad (t,x)\in [0,T]\in\RR_{+}^d  
\end{equation}
and
\begin{equation*}\label{eq:operator-W}
(W_{\rho,c} F)(t,x)=\int_0^t ( e^{-i(t-s)E^r_{\rho,c}} F(s,\cdot)) (x) ds, \quad (t,x)\in [0,T]\times \RR_{+}^d.
\end{equation*}

As in \cite{TBM}, we use norms of the form
$$
\Vert G \Vert_{L^p( [0,T];\mathcal{B})} =
\big \| \Vert G(\cdot) \Vert_{\mathcal{B}} \big \|_{L^p ([0,T])}, \qquad 1\leq p < \infty,
$$
where $\mathcal{B}$ is a Banach space, see also \cite{Taggart}.
In fact, we are interested in the following spaces:
\begin{eqnarray*}
L^1([0,T];\SSS'(\RR^d_+)) & = & \bigcup_{N>0}L^1([0,T];\SSS'_N (\RR^d_+)), \\
L^1([0,T];\mathcal{G}_{0,\alpha}(\RR^d_+)) & = & \bigcap_{h>0}L^1([0,T];\mathcal{G}_{\alpha,h}(\RR^d_+)),\\
L^1([0,T];\mathcal{G}_{\alpha}(\RR^d_+)) & = & \bigcup_{h>0}L^1([0,T];\mathcal{G}_{\alpha,h}(\RR^d_+)),\\
L^1([0,T];\mathcal{G}'_{0,\alpha}(\RR^d_+)) & = & \bigcap_{h>0}L^1([0,T];\mathcal{G}'_{\alpha,h}(\RR^d_+)),\\
L^1([0,T];\mathcal{G}'_{\alpha}(\RR^d_+)) & = & \bigcup_{h>0}L^1([0,T];\mathcal{G}'_{\alpha,h}(\RR^d_+)).
\end{eqnarray*}

Theorem \ref{discon E} indicates that the operator $\ds T_{r,\rho,c}:f\mapsto e^{-itE^r_{\rho,c}}f$ easily becomes discontinuous in the settings of classical function and (ultra)distribution spaces. More precisely,

\begin{thm} \label{thm:propagator_T}
Let there be given $r,T>0$, $\rho,c\in\mathbb{C}$ and let $\operatorname{Im}(\rho)>0$. Then the propagator
$T_{r,\rho, c}$ given by \eqref{eq:operator-T} is discontinuous from
\begin{itemize}
\item[(i)]  $\mathcal{S}(\RR^d_+)$ to $ L^1([0,T];\SSS'(\RR^d_+))$;
\item[(ii)] $ \mathcal G_{0,\alpha} (\RR^d_+)$ 
to $ L^1([0,T];  \mathcal G'_{0,\alpha}(\RR^d_+))$ when $\alpha > 1/r$;
\item[(iii)]
$ \mathcal G_{\alpha} (\RR^d_+)$ to $ L^1([0,T];  \mathcal G'_{\alpha}(\RR^d_+))$ when $\alpha\geq 1/r.$
\end{itemize}
\end{thm}

\begin{proof}
We again restrict our attention to the case 
$r=1$, since the general case follows by analogy.
Moreover, we give the proof for  $\SSS(\RR^d_+)$, since the other cases follow by similar arguments.
As
$$\left(T_{1,\rho,c} f\right)(t,x)=e^{-itE_{\rho,c}}f (x)
= e^{-it(\rho E+c)} f (x)
= e^{-it\rho \left(E+\frac{c}{\rho}\right)}f (x),
\qquad f \in \SSS(\RR^d_+),$$
it is sufficient to prove that for some $t\in[0,T]$ and $f\in\SSS(\RR^d_+)$
we have $(T_{1,\rho,c}) f(t,x)\not\in\SSS'(\RR^d_+)$. 
Since $\operatorname{Re}(-it\rho)=t \operatorname{Im}(\rho)$, it follows that $\operatorname{Re}(-it\rho)>0$, for every $t>0$. The claim now follows 
from  Theorem \ref{discon E} $(i)$.
\end{proof}

On the other hand, Theorem \ref{cont. E} indicates that the propagators are continuous on $P-$spaces when $\alpha>0$ is below the critical value 
$1/r $:

\begin{thm} \label{thm:propagators_T and W}
Let there be given $r,T>0$, and $c\in\mathbb{C}$. Let $\alpha_1,\alpha_2\in \RR_{+}$ be such that $0<\alpha_1 < \frac{1}{r}$ and 
$\alpha_2 \leq \frac{1}{r}$. Then the following is true:
\begin{itemize}
\item[(i)] The operator $T_{r,\rho, c}$ is continuous from 
$$\mathcal{G}_{0,\alpha_1}(\RR^d_+),\mathcal{G}_{\alpha_2}(\RR^d_+), \mathcal{G}_{\alpha_2}'(\RR^d_+), \;\; \text{and}\;\; \mathcal{G}_{0,\alpha_1}'(\RR^d_+)
$$
to the spaces 
\begin{eqnarray*}
L^1([0,T];\mathcal{G}_{0,\alpha_1}(\RR^d_+)), \quad L^1([0,T];\mathcal{G}_{\alpha_2}(\RR^d_+)),\\
L^1([0,T];\mathcal{G}_{\alpha_2}'(\RR^d_+)), \qquad \text{and} 
\qquad
L^1([0,T];\mathcal{G}_{0,\alpha_1}'(\RR^d_+)),
\end{eqnarray*}
respectively.
\item[(ii)] The propagator $W_{r, \rho, c}$ is continuous  on the spaces
\begin{eqnarray*}
L^1([0,T];\mathcal{G}_{0,\alpha_1}(\RR^d_+)), \quad L^1([0,T];\mathcal{G}_{\alpha_2}(\RR^d_+)),\\
L^1([0,T];\mathcal{G}_{\alpha_2}'(\RR^d_+)), \qquad \text{and} 
\qquad
L^1([0,T];\mathcal{G}_{0,\alpha_1}'(\RR^d_+)).
\end{eqnarray*}
\end{itemize}
\end{thm}

\begin{proof}
We comment just the case $(i)$  for $ \mathcal{G}_{\alpha_2}(\RR^d_+),  $ since the other cases can be proved in  a similar way.  Once again, we consider the case $r=1$ only. We will focus on the operator $E_{\rho,c}$, since the continuity properties of $W_{\rho,c}$ can be derived from its definition and the estimates for $T_{1,\rho,c}$.

Let $B$ be a bounded set in $\mathcal{G}_{\alpha_2}(\RR^d_+)$. 
Since $\mathcal{G}_{\alpha_2}(\RR^d_+)$ is a $(DFS)-$space, it follows that 
there exists $h>0$ such that we have
$$ \sup_{f\in B} |a_n(f)|e^{h |n|^{1/\alpha_2}}<\infty, \qquad 
\forall n\in\NN_0^d, \qquad f \in B.$$
Let $\ds M=\sup_{t\in [0,T]} \left| e^{-it\rho c}\right|.$ Then for the arbitrary $t\in [0,T]$ and $f\in B$, we obtain
$$
\left|a_n\left(e^{-itE_{\rho,c}}f\right)\right|\leq M e^{t\operatorname{Im}(\rho) |n|-h|n|^{1/\alpha_2}}  \leq M_1 e^{T\operatorname{Im}(\rho) |n|-h|n|^{1/\alpha_2}}  \leq M_2 e^{-h_1|n|^{1/\alpha_2}  }
$$
for arbitrary $h_1\in [0,h)$ and suitable constants $M, M_1, M_2>0$. Therefore, 
$$\|e^{-it\rho E_{\rho,c}} f\|_{L^1([0,T];\mathcal{G}_{\alpha,h_1}(\RR^d_+))}\leq M_2 T^d, \ \mbox{for every}\ f\in B,$$
which implies that the set 
$$\{  e^{-itE_{\rho,c}}f \;\; : \;\; f \in B \} $$ 
is bounded in $L^1([0,T];\mathcal{G}_{\alpha,h_1}(\RR^d_+))$, wherefrom it follows that it is bounded in $L^1([0,T];\mathcal{G}_{\alpha_2}(\RR^d_+))$ as well. Thus the propagator 
$T_{1,\rho,c}$  which maps a bornological space to a locally convex space 
preserves boundedness of sets. Consequently, it is continuous.
\end{proof}

\begin{rem}\label{ill}
By Theorems \ref{thm:propagator_T} and \ref{thm:propagators_T and W}, 
it follows  that for $r=1$, the equation (\ref{E Schrodinger problem}) is well-posed in the setting of $P-$spaces, $\mathcal G_{0,\alpha_1}(\rdp)$ and $ \mathcal G_{\alpha_2}(\rdp)$,  when  $0<\alpha_1 \leq 1$
and $0<\alpha_2 < 1$. 
On the other hand,  when $ \operatorname{Im} \rho >0$,
the equation (\ref{E Schrodinger problem}) is
ill-posed for Schwartz functions on $\RR^d_+$, $G$-type functions (and consequently $P-$spaces above the critical value $1$). 
\end{rem}

\section{Relation to Schr\"{o}dinger type equations with harmonic oscillator} \label{sec5}

As it is shown in \cite[Section 5]{JPTV}, $P-$spaces on positive orthants are  closely related to Pilipovi\'{c} spaces on the whole space through the simple change of variables $f \mapsto \mathcal J^{-1} f$, 
see \eqref{J-intro}. 
In this section we  exploit this connection  to transform the  equation (\ref{E Schrodinger problem}) into a
Schr\"{o}dinger type equation of the form considered in  \cite[(4.18)]{TBM}.
In addition, we provide interpretation of these results 
when the initial values are ultradistributions. To that end we 
utilize transposes and inverses of operators related
to the radial symmetry.

\subsection{The problem with test functions as initial values} 

In the sequel, we simplify the equation (\ref{E Schrodinger problem}) by taking $r=1$, and focus on the homogeneous $1-$dimensional case, i.e. $d=1$ and $F(t,x)\equiv 0$.

Let $E_{\rho,c}$ be given by  \eqref{eq:E_ro_ce}, $\rho\in\CC \setminus \{ 0\}$,  $c\in\CC$. 
The change of variables  $x=r^2$  applied to the Cauchy problem
\begin{equation}\label{1dproblem2}
\begin{array}{ll}
         & i\partial_t u(t,x)-E_{\rho,c} u(t,x)=0,\quad (t,x)\in [0,T] \times \RR_+,
         \\
         & u(0,x)=f(x),  \end{array}  
\end{equation} 
yields to
\begin{equation}
\label{1dproblem3}
 \begin{array}{ll}
         & i\partial_t v(t,r)-J_{\rho, c} v(t,r)=0, \quad(t,r)\in [0,T] \times \RR,
         \\
         & v(0,r)=g(r), \end{array} 
\end{equation}
where    $v(t,r)=u(t,r^2)$, $g(r)=f(r^2)$ and the operator $J_{\rho, c}$ 
is given by
\begin{equation} \label{operatorJrc}
J_{\rho, c} = -\frac{\rho}{4}
\left (  \frac{d^2}{d r ^2} + \frac{1}{r}  \frac{d}{d r}
 - r^2\right ) + c-\frac{\rho}{2},
\end{equation}
i.e.   $J_{\rho, c}$ 
acts on a function $w=w(r)$ in the following way:
\begin{equation}\label{operatorJ}\left(J_{\rho,c}w\right) (r)=-\rho\frac{w''(r)}{4}-\rho\frac{w'(r)}{4r}+\rho\frac{r^2}{4}w(r)  -\left(\frac{\rho}{2}- c\right)w(r), \quad r \in \RR. \end{equation}
Therefore, for suitable function $u\in\SSS(\RR_{+})$  and the function $v$ defined  on $\RR$ by $v(r)=u(r^2) $ we have
\begin{equation}\label{JE}
(J_{\rho,c} v)(r)= ( E_{\rho,c} u)(r^2)  = (E_{\rho,c} u \circ \rho )(r) =  (E_{\rho,c} u)(\rho (r)), \qquad r\in\RR.
\end{equation}
Note that the function $g$ in 
\eqref{1dproblem3} 
is even. Hence, the same holds for  $x\mapsto v(t,x)$, for each $t\in [0,T]$. 

The converse relations can also be proved. Namely, 
equation
(\ref{1dproblem3}) 
becomes 
(\ref{1dproblem2})
after applying substitution $x=r^2$.

\par

Next, we recall that $ f \in \mathcal H_{\alpha/2,even}(\RR ) $
($ f \in \mathcal H_{0,\alpha/2,even}(\RR ) $), $ \alpha > 0$,
if \eqref{eq:even} holds.

From the previous considerations and utilizing  \\ 
\cite[Theorem$5.5$, Theorem$5.7$]{JPTV} we have the obtain the following proposition. 

\begin{prop}\label{kvadrat}
Let $E_{\rho,c}$ be given by  \eqref{eq:E_ro_ce}, $\rho\in\CC \setminus \{ 0\}$,  $c\in\CC$,  and let $\alpha>0$.  Then the equation \eqref{1dproblem2} with the initial value $f\in \mathcal G_{\alpha}(\RR_{+})$ is transformed 
into the equation \eqref{1dproblem3} with the initial value $g(r)=f(r^2)\in \mathcal H_{\alpha/2,radial}(\RR ).$ 

Conversely, for  initial value $g\in \mathcal H_{\alpha/2,radial}(\RR)$  equation \eqref{1dproblem3} transforms into the equation \eqref{1dproblem2} via  the substition $x=r^2$ with  initial value  $f(x)=g(\sqrt{x})\in\mathcal G_{\alpha}(\RR)$.

The same holds if  $ \mathcal G_{\alpha}(\RR)$ is replaced by 
 $\mathcal G_{0,\alpha}(\RR)$, and $\mathcal H_{\alpha/2,radial}(\RR )$
 is replaced by $\mathcal H_{0,\alpha/2,radial}(\RR )$ at each occurence.
\end{prop}

For the sake of simplicity, from now on we considering the Roumieu case. The Beurling case can be treated in a similar way, and details are left for the reader.

We continue our investigations by establishing a connection between the operator
$J_{\rho, c}$ as given in  \eqref{operatorJrc} and the harmonic oscillators considered in \cite{TBM}. 

We slightly modify the notation from \cite[(1.7)]{TBM}  for $d=2$, and define the operator
\begin{eqnarray} \label{eq:harm-osc}
 H_{\rho,c} & = & \rho H + c
=
\rho \left(-\Delta +|x|^2 \right) + c \\
& = &
\rho \left(-\frac{\partial ^2}{\partial ^2 x_1}-\frac{\partial^2}{\partial ^2 x_2}+x_1^2+x_2^2\right)+ c, 
\qquad x = (x_1, x_2) \in \RR^2,  \nonumber
   \end{eqnarray}
$\rho \in\CC \setminus \{ 0 \},$ $,c\in\CC,$ so that $H= H_{1,0}$.
The correspoding propagator $ e^{-it  H_{\rho,c}}$ is given by
\begin{equation} \label{eq:HarmOscProp}
e^{-it  H_{\rho,c}} f  = \sum_{n\in\mathbb{N}^d_0} 
e^{-2 it  \rho (|n| +1 + \frac{c}{\rho} )}
c_n(f) h_n
\end{equation}
where the coefficients $c_n (f)$ are given by \eqref{eq:Hermitecoeff},
cf. \cite[(2.3)]{TBM}.

It is straightforward to verify that if
$g$ is smooth enough and 
defines  the radial function 
$$
\modu{g}(x) = \modu{g}(x_1,x_2) = (\Phi g) (x_1,x_2)
=g(\sqrt{x_1^2+x_2^2}), \qquad  x = (x_1, x_2) \in \RR^2,
$$
then the following identity holds: 
\begin{equation} \label{radial}
\left(J_{\rho,c} \ g\right) (|x|)=
\left ( 
\frac{\rho(-\Delta+ |x|^2)}{4} +(c-\frac{\rho}{2}) 
\right) \modu{g} (x)
=H_{\frac{\rho}{4},c-\frac{\rho}{2}} \modu{g} (x), \qquad x \in \RR ^2,
\end{equation}
$ |x| = \sqrt{x_1^2+x_2^2}$, 
and the operator $J_{\rho,c}$ is given by \eqref{operatorJrc}.

\par

Assume now that $u=u(t,x)$ is a solution to equation (\ref{1dproblem2}) with the initial value $f\in\mathcal G_{\alpha}(\RR_+)$. By Proposition \ref{kvadrat} 
the function $v(t,r)=u(t,r^2)$ satisfies equation (\ref{1dproblem3}) 
with the initial value $g(r)=f(r^2)\in \mathcal H_{\alpha/2,even}(\mathbb R).$ 
Define 
$$
w(t,x_1,x_2):=v(t,\sqrt{x_1^2+x_2^2})=u(t,x_1^2+x_2^2), 
\qquad
(t, (x_1,x_2)) \in [0,T] \times  \RR^2.
$$ 

Then $v(t,r)=w(t,x_1,x_2)$, $ r = \sqrt{x_1^2+x_2^2}$, 
and the equation (\ref{1dproblem3}) becomes
$$
0=\left(i\partial_t -J_{\rho,c}\right) v(t,r)= 
\left(i\partial_t  - H_{\frac{\rho}{4},c-\frac{\rho}{2}}\right) w(t,x_1,x_2),
$$
where we used \eqref{radial}. Therefore,  the following result holds.

\begin{prop} \label{prop:Toft-i-mi}
Let   $f\in \mathcal G_{\alpha}(\RR_{+})$, $\alpha>0$. 
Suppose that the function $u(t,x)$ solves the initial value problem (\ref{1dproblem2}) with  initial value $f$. Then the function $w(t,x_1,x_2)=u(t,x_1^2+x_2^2)$,
$t\in [0,T]$, $x=(x_1,x_2)\in \RR^2$, solves the initial value problem 
\begin{equation}\label{1dproblem4}
 \begin{array}{ll}
         & i\partial_t w(t,x_1,x_2)-H_{\frac{\rho}{4},c-\frac{\rho}{2}} w(t,x_1,x_2)=0,\quad (t,x)\in [0,T] \times \RR^2\\ \medskip
         & w(0,x_1,x_2)=f(x_1^2+x_2^2).  \end{array}  
\end{equation} 
Moreover, the initial data $(x_1,x_2)\mapsto f(x_1^2+x_2^2)$ belongs to $\mathcal H_{\frac{\alpha}{2},even}(\RR^2)$. 
\end{prop}

Before stating the converse of Proposition \ref{prop:Toft-i-mi}, we need some preparation.

Let $\modu{g}\in\mathcal H_{\frac{\alpha}{2},radial}(\RR^2)$, $\alpha>0$. By Lemma \ref{radial2} the function $g$ defined by 
$g=\Phi^{-1}(\modu{g})$
belongs to the  space $\ds \mathcal H_{\frac{\alpha}{2},even}(\RR)$. Moreover, from \cite[Theorem 5.5]{JPTV} it follows that 
the function $f$  defined by 
$f(x_1^2+x_2^2)=\modu{g}(x_1,x_2)$ belongs to the space $\mathcal G_{\alpha}(\RR_{+})$.

\par

By identifying $\modu{g}(x_1,x_2)$ with $f(x_1^2+x_2^2)$ in \eqref{1dproblem4}, 
we obtain the following initial value problem: 
\begin{equation}\label{1dproblem4+}
 \begin{array}{ll}
         & i\partial_t w(t,x_1,x_2)-H_{\rho,c} w(t,x_1,x_2)=0,\quad (t,(x_1,x_2))\in [0,T] \times \RR^2,\\
         & w(0,x_1,x_2)=\modu{g}(x_1,x_2)  \end{array}  
\end{equation} 
whose solution is given by
$\ds w(t,x)=e^{-itH_{\rho,c}}\modu{g}$. Moreover, from (\ref{operatorJ}) and (\ref{radial}) we obtain 
\begin{equation}\label{fg}
\left(E_{4\rho,c+2\rho} f \right )(x_1^2+x_2^2)
=\left(J_{4\rho,c+2\rho} g\right)
(\sqrt{x_1^2+x_2^2})
=\left(H_{\rho, c}\modu{g}\right)(x_1,x_2), 
\end{equation}
$(x_1, x_2) \in \RR ^2 $.

Next, we consider
\begin{equation}\label{J}
\mathcal{J}(\psi)=\psi|_{\RR^d}\circ \varrho,\quad \mathcal J^{-1}(\varphi)=\varphi\circ \rho,
\end{equation}
where  $\psi \in\mathcal H_{\alpha/2,even}(\RR^d)$, and
$\varphi\in \mathcal G_{\alpha}(\RR^d_{+})$, $\alpha > 0$, cf. \eqref{J-intro}.
By \cite[Section 3]{SSSB} and \cite[Corollaries  5.6 and 5.8]{JPTV} it follows that the mapping 
$$
\mathcal J:\mathcal H_{\alpha/2,even}(\RR^d)\to \mathcal G_{\alpha}(\RR^d_{+}), 
$$
is a topological isomorphism. 
We now rewrite  \eqref{JE} 
in terms of $\mathcal{J}$ and $\mathcal{J}^{-1}$:
\begin{equation}\label{eq:JE}
J_{\rho,c} \phi = \left(\mathcal J^{-1}\circ E_{\rho,c}\circ \mathcal{J}\right)\phi, \qquad 
\phi \in\mathcal{S}(\RR^d),
\end{equation}
with $J_{\rho,c}$ and $E_{\rho,c}$ given by \eqref{operatorJrc} and 
\eqref{eq:E_ro_ce} respectively. 

Moreover, by using \eqref{eq:Phi}, it follows that 
\eqref{fg} can be written as
\begin{equation}\label{eq:he}
H_{\rho,c}=\mathcal{J}^{-1}\circ \Phi\circ E_{4\rho,c+2\rho} \circ \mathcal{J}\circ \Phi^{-1}\quad \mbox{on}\quad  
H_{\alpha/2,radial} (\RR^d),
\quad \alpha>0.
\end{equation}

The following diagram illustrates the relationships between $\mathcal J$, $\Phi$ and their corresponding ranges.

\begin{center}
    \begin{tikzpicture}[>=stealth]
\node (A) at (0,2)
{$\mathcal H_{\alpha/2,\mathrm{even}}(\mathbb R)$};

\node (B) at (5,2)
{$\mathcal H_{\alpha/2,\mathrm{radial}}(\mathbb R^2)$};

\node (C) at (2.5,0)
{$\mathcal G_{\alpha}(\mathbb R_{+})$};

\draw[->] (A) -- node[above] {$\Phi$} (B);
\draw[->] (A) -- node[left] {$\mathcal J$} (C);
\draw[->] (B) -- node[right] {$\mathcal J \circ \Phi^{-1}$} (C);
\end{tikzpicture}
\end{center}

Now the we have the following.

\begin{prop} \label{rot-inv-novo}
Let 
$\modu{g}\in \mathcal H_{\frac{\alpha}{2},even}(\RR^2)$, $\alpha > 0$,
be  a radially symmetric function. 
Let $w(t,x)=w(t,x_1,x_2)$ denote the solution of the initial value problem (\ref{1dproblem4+}) 
with initial value $\modu{g}$. Then the function $u$  defined by 
$$
u(t,x_1^2+x_2^2)=w(t,x_1,x_2), \qquad (t,  (x_1,x_2))
\in  [0,T]\times\RR^2,
$$ 
solves the initial value problem 
\begin{equation}
\label{1dproblemE}
 \begin{array}{ll}
         & i\partial_t u(t,x)-E_{4\rho,c+2\rho} u(t,x)=0,\quad (t,x)\in [0,T] \times \RR_+\\
         & u(0,x)=\modu{g}(x,0).  \end{array} 
\end{equation} 
with the initial value  $x\mapsto \modu{g}(\sqrt{x},0) \in \mathcal G_{\alpha}(\RR_{+}).$

\end{prop}

\begin{proof}

Fix $t\in [0,T]$, and put $r=\sqrt{x_1^2+x_2^2}$, 
$x_1,x_2\in\RR$. 
By \cite[(1.37)]{TBM} the solution $w(t,x)$ of the initial value problem (\ref{1dproblem4+}) is calculated as
$w(t,x_1,x_2)=
\left(e^{-itH_{\rho,c}}\modu{g}\right)$. 
From  \eqref{eq:he} we get 
$$
H_{\rho,c}=(\mathcal{J}\circ \Phi^{-1})^{-1}\circ E_{4\rho,c+2\rho} \circ \mathcal{J}\circ \Phi^{-1},
 $$
so that 
\begin{equation}\label{eq:J-Phi-H-E}
e^{-itH_{\rho,c}}=(\mathcal{J}\circ \Phi^{-1})^{-1}\circ e^{-itE_{4\rho,c+2\rho}}\circ  \mathcal{J}\circ \Phi^{-1} .\end{equation}
We also note that 
\begin{equation} \label{eq:J^{-1}-Phi-commute}
\mathcal{J}^{-1}\circ \Phi=\Phi\circ\mathcal{J}^{-1}\quad \text{on}\quad  \mathcal G_{\alpha}(\RR),\quad \alpha>0. 
\end{equation}
Let $g=\Phi^{-1}(\modu{g})\in \mathcal H_{\alpha/2,even}(\RR)$, $f=\mathcal{J}g\in \mathcal G_{\alpha}(\RR_{+}).$
The solution of the initial value problem \eqref{1dproblemE} is given by $u(t,\cdot)=e^{-itE_{4\rho,c+2\rho}}f.$
Now we have 
$$e^{-itH_{\rho,c}}\modu{g}=\left((\mathcal{J}\circ \Phi^{-1})^{-1}\circ e^{-itE_{4\rho,c+2\rho}}\right)f.$$
Since
$$\Phi\left( (\mathcal{J}\circ e^{-itE_{4\rho,c+2\rho}}) f\right)= w(t,\cdot),$$
it follows that
$
e^{-itH_{\rho,c}}\modu{g}\in \mathcal H_{\alpha/2,radial}(\RR^2)$,  for every $t\in [0,T]$, 
and 
$$(\mathcal{J}\circ \Phi^{-1})w(t,\cdot)= u(t,\cdot), \quad \text{i.e.} \quad 
w(t,x_1,x_2)=u(t,x_1^2+x_2^2),  \qquad (x_1,x_2)\in\RR^2.
$$
\end{proof}

\begin{cor}
 When  $ \operatorname{Im} \rho >0$ and $\alpha\geq 1$,  Remark \ref{ill} claims that the initial value problem (\ref{1dproblem2}) is ill posed.  Moreover, the initial value problem (\ref{1dproblem4}) is also ill posed in view of \cite[Corollary 4.11]{TBM}. When $\alpha\in (0,1)$ both problems are well posed.  
\end{cor}

\subsection{The problem with distributional initial values} \label{subsec:5.2}

Next, we discuss the situation when the initial value of the problem 
\eqref{1dproblem2}
belongs to a suitable space of ultradistributions. To that end, we need a proper interpretation of 
the change of variables when considering the dual spaces. 

Recall that $\mathcal J$ given by \eqref{J} defines a topological isomorphism between 
$\mathcal H_{\alpha/2,even}(\RR^d)$ and $\mathcal G_{\alpha}(\RR^d)$,
and between $\mathcal H_{0,\alpha/2,even}$ and $ \mathcal G_{0,\alpha}(\RR^d) $.
Also, its transposed counterpart gives rise to topological isomorphisms on the corresponding dual spaces, i.e. 
$$\mathcal J':\mathcal G'_{\alpha}(\RR^d)\rightarrow \mathcal H'_{\alpha/2,even}(\RR^d) , 
\quad \text{and} \quad
\mathcal J':\mathcal G'_{0,\alpha}(\RR^d) \rightarrow \mathcal H'_{0,\alpha/2,even}(\RR^d). 
$$ 
Recall that
$$
J_{\rho,c} \phi = \left(\mathcal J^{-1}\circ E_{\rho,c}\circ \mathcal{J}\right)\phi, \qquad 
\phi \in\mathcal{S}(\RR^d),
$$
with $J_{\rho,c}$ and $E_{\rho,c}$ given by \eqref{operatorJrc} and 
\eqref{eq:E_ro_ce} respectively.

The continuity properties of $E_{\rho,c}$  and 
$ J_{\rho,c}$, $\rho\in\mathbb C \setminus \{ 0\}$, $ c \in \mathbb C$,
imply the following result, see also \cite[Theorem 3.4]{SSSB}.

\begin{lem} Let there be given $\rho\in\mathbb C \setminus \{ 0\}$, $ c \in \mathbb C$, and $\alpha>0$. Then the mapping 
$$
E_{\rho,c}:\mathcal G_{\alpha}(\RR^d)\to \mathcal G_{\alpha}(\RR^d) \qquad   (E_{\rho,c}:\mathcal G_{0,\alpha}(\RR^d)\to \mathcal G_{0,\alpha}(\RR^d))
$$ 
is continuous if and only if 
$$
J_{\rho,c}:\mathcal H_{\alpha/2,even}(\RR^d)\to \mathcal H_{\alpha/2,even}(\RR^d) \qquad (J_{\rho,c}:\mathcal H_{0,\alpha/2,even}(\RR^d)\to \mathcal H_{0,\alpha/2,even}(\RR^d))
$$ 
is continuous.
Moreover, the transposed mapping
$$
J'_{\rho,c}:\mathcal H'_{\alpha/2,even}(\RR^d)\to \mathcal H'_{\alpha/2,even}(\RR^d) \;  (
J'_{\rho,c}:\mathcal H'_{0,\alpha/2,even}(\RR^d)\to \mathcal H'_{0,\alpha/2,even}(\RR^d)) $$ 
is  continuous as well.
\end{lem}

Since we are interested in the dual counterpart of  \eqref{JE}, we proceed below with $d=1$.

Let there be given $t\in\RR\setminus\{0\}$. From \eqref{eq:JE} 
and the fact that  the operator $E_{\rho,c}$ is self-adjoint,
we conclude that
\begin{equation*}\label{eq:ej}
e^{-itJ_{\rho,c}}=\mathcal{J}^{-1}\circ e^{-itE_{\rho,c}}\circ \mathcal {J} \end{equation*}
on  $\mathcal H_{\alpha/2}(\RR)$ and $\mathcal H_{0,\alpha/2}(\RR)$, 
with the dual counterpart
\begin{equation*}\label{eq:ejdual}
e^{-itJ'_{\rho,c}}=\mathcal{J'}\circ e^{-itE_{\rho,c}}\circ \mathcal {(J'})^{-1} 
\end{equation*}
on $\mathcal H'_{\alpha/2}(\RR)$   and $\mathcal H'_{0,\alpha/2}(\RR)$.

 The solution  of the Cauchy problem \eqref{1dproblem2} with the initial value $f\in\mathcal G'_{\alpha}(\RR_{+})$ is formally given by 
 $u(t, \cdot)=e^{-itE_{\rho,c}}f$, $ t \in [0,T]$. 
Note that $g=\mathcal{J}'f\in \mathcal H'_{\alpha/2,even}(\RR)$. 
Now let $v_t (\cdot) = v(t,\cdot )=e^{-itJ'_{\rho,c}}\left(\mathcal{J}'f\right)$,
$t \in [0,T]$. Thus we obtain the family of generalized functions  
$$
v_t (\cdot) = v(t,\cdot )=e^{-itJ'_{\rho,c}}\left(\mathcal{J}'f\right)=\left(\mathcal{J'}\circ e^{-itE_{\rho,c}}\right)f
$$ 
wherefrom
$
(\mathcal{J'})^{-1} v(t,\cdot ) = e^{-itE_{\rho,c}}f=u(t,\cdot ),$
so that the family of ultradistributions $v_t(\cdot)=v(t,\cdot)=\mathcal{J}'u(t,\cdot) $ 
solves the initial value problem 
\begin{equation}
\label{1dproblem5}
 \begin{array}{ll}
         & i\partial_t v(t,r)-J'_{\rho, c} v(t,\cdot )=0, \quad t\in[0,T]
         \\
         & v(0,\cdot )=g, \end{array}  
\end{equation}

Therefore the following is true.

\begin{prop} \label{prop:dist-in-value-1}
Let $\alpha>0$, $\rho\in\CC\setminus\{0\},c\in\mathbb C$. 
If  the family of ultradistributions $u(t,\cdot)$ solves the  Cauchy problem \eqref{1dproblem2} with the initial data $f\in\mathcal G'_{\alpha}(\RR_{+})$, ($f\in\mathcal G'_{\alpha}(\RR_{+})$) then the family of ultradistributions $v(t,\cdot)=\mathcal{J}'u(t,\cdot)$ solves the initial value problem \eqref{1dproblem5} with the initial value $g=\mathcal{J}'f\in\mathcal H'_{\alpha/2,even}(\RR)$  ($g=\mathcal{J}'f\in\mathcal H'_{0,\alpha/2,even}(\RR)$).

If, on the other hand, the family of ultradistributions $v(t,\cdot)$ solves the Cauchy problem \eqref{1dproblem5} with the initial datum $g\in \mathcal H'_{\alpha/2,even}(\RR) $  ($g\in \mathcal H'_{0,\alpha/2,even}(\RR) $) then the family of ultradistributions $u(t,\cdot)=({\mathcal J'})^{-1}v(t,\cdot)$ solves the Cauchy problem \eqref{1dproblem2} with the initial data $f\in ({\mathcal J'})^{-1}g\in \mathcal G_{\alpha}(\RR_{+}) $ (in the space $\mathcal G_{0,\alpha}(\RR_{+}) )$.
\end{prop}

Finally, we  present dual counterparts of Propositions  \ref{prop:Toft-i-mi} and \ref{rot-inv-novo}. Let $\rho>0$. 
Then the dual counterpart of \eqref{eq:he} is given by


\begin{equation*}\label{eq:he-dual}
H_{\rho,c}=(\Phi')^{-1}\circ \mathcal{J}'\circ E_{4\rho,c+2\rho} \circ \Phi'\circ (\mathcal{J'})^{-1} \quad \mbox{on}
\quad
\mathcal G'_{0,\alpha}(\RR_{+}), \alpha>0.
\end{equation*}
Note that for $f\in \mathcal{S}(\RR_{+})$ and $ r>0$ we have
$$
\left((\mathcal{J}^{-1}\circ \Phi\right)f)(r)= (\mathcal{J}^{-1}f)(\sqrt{x_1^2+x_2^2})=f(x_1^2+x_2^2)=\left( (\Phi\circ\mathcal{J}^{-1})f\right)(r),
$$
cf. \eqref{eq:J^{-1}-Phi-commute},
which also holds for the transpose mappings, and therefore 
\begin{equation}\label{eq:komutativnost-dual}
   (\mathcal{J}')^{-1}\circ \Phi'=\Phi'\circ(\mathcal{J})'^{-1}\ \mbox{on}\  \mathcal G'_{\alpha}(\RR_{+}),\ 
(\text{ and on} \;\;\;  \mathcal G'_{0,\alpha}(\RR_{+})),  \;\; \;\alpha>0 . 
\end{equation}
Let us consider the problem \eqref{1dproblem2} with the initial value $f\in \mathcal  G'_{\alpha}(\RR_{+})$ and the (formal) 
solution $u(t,\cdot)=e^{-it H_{\rho,c}}f$, $ t \in [0,T]$. It now follows that $g=\mathcal{J}'f\in \mathcal H'_{\alpha/2,even}(\RR)$.
From  Proposition \ref{prop:phi'} and commutativity relation
\eqref{eq:komutativnost-dual} 
we conclude that
$$
\modu{g}=\left((\Phi')^{-1} \circ \mathcal{J}'\right)f=\left(\mathcal{J}'\circ (\Phi')^{-1}  \right)f
\in \mathcal H'_{\alpha/2,radial}(\RR^2).
$$
Let $w(t,\cdot)=e^{-itH_{\frac{\rho}{4},c-\frac{\rho}{2}} }\modu{g} $,
$t\in [0,T]$.
 Since 
 $$
 H_{\frac{\rho}{4},c-\frac{\rho}{2}}=((\mathcal{J}^{-1})'\circ \Phi')^{-1} \circ E_{\rho,c} \circ ((\mathcal{J}^{-1})'\circ \Phi')
 $$
 we get
 $$
 e^{-itH_{\frac{\rho}{4},c-\frac{\rho}{2}} }=(\Phi')^{-1}\circ \mathcal{J}'\circ e^{-itE_{\rho,c}}\circ \Phi'\circ (\mathcal{J'})^{-1}, 
 $$
 and therefore 
\begin{eqnarray*}
w(t,\cdot) & = & e^{-itH_{\frac{\rho}{4},c-\frac{\rho}{2}} }\modu{g}=\left((\Phi')^{-1}\circ \mathcal{J}'\circ e^{-itE_{\rho,c}}\circ \Phi'\circ (\mathcal{J'})^{-1}\circ \mathcal{J}'\circ (\Phi')^{-1}\right)f \\
& = & 
\left((\Phi')^{-1}\circ \mathcal{J}'\circ e^{-itE_{\rho,c}}\right)f.
\end{eqnarray*}

Finally we have 
$$
  \left(\Phi'\circ (\mathcal{J}')^{-1}\right)w(t,\cdot)= \left((\mathcal{J}')^{-1}\circ \Phi'\right) w(t,\cdot)=e^{-it H_{\rho,c}}f=u(t,\cdot),
$$
which gives 
$w(t,\cdot)=\left(\mathcal{J}'\circ (\Phi')^{-1}  \right) u(t,\cdot)$.
The mappings $\Phi'$ and $\mathcal{J}'$ are bijections (moreover, topological isomorphisms) on suitable spaces, so we have the following
result concerning the dual spaces.

\begin{prop} \label{prop:dist-in-value-2}
 Let $\alpha>0$.  If $u(t,\cdot)$ solves the Cauchy problem  \eqref{1dproblem2} with initial data $f\in \mathcal G'_{\alpha}(\RR)$ (  $f\in \mathcal G'_{0,\alpha}(\RR)$), 
 then the family of ultradistributions 
 $$
 w_t (\cdot) = w(t,\cdot)=\left(\mathcal{J}'\circ (\Phi')^{-1}  \right) u(t,\cdot), \qquad t \in [0,T],
$$  
solves the Cauchy problem \eqref{1dproblem4}
with the initial value  $\modu{g}=\left(\mathcal{J}'\circ (\Phi')^{-1}  \right)f
\in \mathcal H'_{\alpha/2,radial}(\RR^2)$ ( $\modu{g}
\in \mathcal H'_{0,\alpha/2,radial }(\RR^2)$).

If, on the other hand, the family $w_t (\cdot)$, $t \in [0,T]$,
solves the Cauchy problem \eqref{1dproblem4+} with initial value 
$\modu{g}\in  \mathcal H'_{\alpha/2,radial}(\RR^2)$
($\modu{g}\in  \mathcal H'_{0,\alpha/2,radial}(\RR^2)$), then the family $u(t,\cdot)=(\Phi'\circ(\mathcal{J}')^{-1})w(t,\cdot)$ solves the initial value problem 
\eqref{1dproblem4} with the initial value $f=(\Phi'\circ(\mathcal{J}')^{-1})\modu{g}\in \mathcal G'_{\alpha}(\RR_{+})$ ($f\in \mathcal G'_{0,\alpha}(\RR_{+}))$.
\end{prop}

\section{Relations between fractional transforms} \label{sec6}

By combining the formulas derived in the previous sections with the calculations presented below, we establish relations between several fractional integral transforms. For example, representations of some fractional transforms in terms of the propagators of operators on related domains are given in Table \ref{Table_2} (see also \cite{Atanas}.

{ \small
\begin{table}[!h]{
		\centering
		\begin{tabular}{|l | c | c| l| }\hline
transform		   & representation  &  propagator of & domain  \\ \hline
frac. Fourier 	   &  ${\mathcal F}_{\rho} =  e^{i\frac{\pi \rho}{2}}e^{-\frac{i\pi}{4}H_{\rho,0}}$   &  Harm. oscillator   & $\mathcal H'_{\alpha/2,radial}(\RR^2)$ \\ \hline
frac. pow. of H-C &    $\mathcal{I}_{\mathbf{e^{it}},0} =e^{it E}$ &  Laguerre operator & $\mathcal G'_{\alpha}(\RR_{+})$ \\ \hline
frac. Hankel &   $\mathscr{H}_t=e^{itJ}$    &  the $J$ operator 
& $\mathcal H'_{\alpha/2,even}(\RR)$  \\ \hline
   \end{tabular}
	}
	\caption{Relation between transforms and propagators  }\label{Table_2}
\end{table}
}

For reader's convenience, we recall the definition of fractional Fourier transform, see \cite[the equation (0.5)]{TBM}. Let  $f\in \SSS(\RR^d)$ and $t\in\RR$. 
Then we define fractional Fourier transform acting on $f$ in the following way:
\begin{equation}\label{eq:frac-fourier}
 \mathcal F_{\rho} f(\xi)=\langle  K_{d,\rho}(\xi,\cdot ),f \rangle     
\end{equation}
where $\displaystyle K_{d,\rho}=\bigotimes_{j=1}^d K_{\rho}$ and $K_{\rho}$ is defined as follows. 
$$K_{\rho}(\xi, x)=
\begin{cases}
\left( \frac{1-\cot(\frac{\pi \rho}{2})}{2\pi}\right)^{1/2}  e^{i\cdot \frac{(x^2+\xi^2)\cos\frac{\pi \rho}{2}-2x\xi}{2\sin \frac{\pi \rho}{2} }}   , & \rho\in\RR, \frac{\rho}{2}\not \in\ZZ, \\
\delta (\xi-x) , & \frac{\rho}{4}\in\ZZ, \\
\delta(\xi+x), &  \frac{\rho-2}{4}\in\ZZ. \\
\end{cases}$$

In particular, $\mathcal F_{1}$ and  $\mathcal F_{3}$ are
the Fourier transfom and the inverse Fourier transform, respectively.

 Let now $\rho\in\RR\setminus\{0\}$ and $c=0$. We will use the relation   \cite[Theorem 3.2]{TBM}
  between the propagator of harmonic oscillator
   and the fractional Fourier transform on $\RR^d:$
$$ e^{-i\frac{\pi}{4}H_{\rho,0}}=e^{-i\frac{\pi\rho d}{4}}{\mathcal F}_{\rho}.$$

\par

We now proceed to justify the results presented in Table  \ref{Table_2} 
 and discuss their interpretation when extended to ultradistributions.

Let  $\rho\in\CC\setminus\{0\},c\in\CC$.
Note that 
$$e^{-itJ_{\rho,c}}=\mathcal{J}^{-1}\circ e^{-itE_{\rho,c}}\circ \mathcal {J} $$
on $\SSS(\RR)$.
Thus,
for every even  function $g$ such that $ f(x)= 
\mathcal {J} (g) (x)
= g(\sqrt{x}) $
admits  Laguerre expansion of the form
$f = \sum_{n\in\NN_0} a_n(f) l_n$
we yield
$$
e^{ J_{\rho,c}} g (x) = \sum_{n\in\NN_0} a_n(f) e^{\rho n+c} l_n(x^2), 
\qquad x \in \RR.
$$ 
When $\rho = 1$ and $c = 0$, we set $e^{-it J_{1,0}}=e^{-itJ}$.

We follow the approach from \cite[Section 3, eq. (3.4)]{Namias} and define the fractional Hankel transform 
$\mathscr{H}_{t}$, $t \in \RR\setminus\{0\}$  
as follows:
 \begin{equation}\label{eq:frachankel}
 \mathscr{H}_{t} g (x)=\sum_{n=0}^{\infty} a_n(f) e^{in t} l_n(x^2)=\left(\left(\mathcal{J}^{-1}\circ e^{itE}\circ\mathcal{J}\right)g \right)(x), \qquad 
 x \geq 0,
 \end{equation}
$g\in\SSS(\RR_{+})$ and $f = \mathcal{J}g$.
We remark that the fractional Hankel transform can 
also be introduced in a different way, see for example \cite{Kerr},
and that the choice $t=\pm \pi$ reveals  the classical  Hankel transform, cf. \cite{Namias}.

We proceed with the Roumieu case, and note that the same conclusions hold for the Beurling case.

For  $g\in\mathcal H_{\alpha/2,even}(\RR)$ we have $f= \mathcal{J}g\in \mathcal G_{\alpha}(\RR_+)$  and the following equalities hold:
 \begin{equation}\label{eq:frac1}
 \mathscr{H}_{t} g(x)= 
 \left(e^{it J} g\right)(x)=(e^{it E} f)(x^2)=
 \left(\mathcal{I}_{e^{it},0}f\right)(x^2),\qquad x \in \RR.
 \end{equation}

By using \eqref{eq:frachankel} and the identity 
$$
\int_{\RR_{+}}l_n(x^2)l_m(x^2)xdx=\frac{1}{2}\delta_{m,n},
$$
($\delta_{m,n}$ denotes the Kronecker delta) we deduce 
\begin{equation}\label{eq:Hankeldual}
\int_{\RR_{+}} (\mathscr{H}_{t} g)(x^2)  \phi(x^2) xdx
= \int_{\RR_{+}} g(x^2) (\mathscr{H}_{t}\phi)(x^2)xdx=\langle g,\mathscr{H}_{t}\phi\rangle,
\end{equation}
when  $ g,\phi\in \SSS_{even}(\RR)$. 
We further extend \eqref{eq:Hankeldual} to the duality relation as follows.

Let $g\in \mathcal H'_{\alpha/2,even}(\RR)$. 
We then define the ultradistribution  $\mathscr{H}_{t} g$,  $t\in\RR\setminus\{0\}$, 
in the following way:
$$(\mathscr{H}_{t}g)(\varphi)=\langle g, \left(\mathcal{J}^{-1}\circ e^{-itE}\circ\mathcal{J}\right)  \varphi) \rangle,
\qquad 
\varphi\in \mathcal H_{\alpha/2,even}(\RR),
$$
i.e. $$\mathscr{H}_{t}g=\left(\mathcal{J}'\circ e^{-itE}\circ (\mathcal{J'})^{-1} \right)g.$$

Let $f=\mathcal {(J'})^{-1}(g)\in\mathcal G'_{\alpha}(\RR).$
Then, by using 
 \eqref{hcdistribution} we obtain the dual counterpart of \eqref{eq:frac1}:
\begin{equation}\label{eq:fracdual1}
\mathscr{H}_{t} g=\left(\mathcal{J'}\circ  \mathcal{I}_{e^{it},0}\right)(f). \end{equation}


We proceed by establishing connections between 
the Laguerre operator propagator, the fractional Fourier transform
and the fractional power of Hankel-Clifford transform.

Let $d=2$,  and  $\modu{g}\in \mathcal H_{\frac{\alpha}{2},radial}(\RR^2)$, $\alpha>0$. 
From Proposition \ref{Phi} we have  $g=\Phi^{-1} \modu{g}\in {\mathcal  H}_{\alpha/2,even}(\RR). $
Moreover, the function $f$ defined by $f=\mathcal{J}g = \mathcal{J} \circ \Phi^{-1} \modu{g} $ belongs to
$\mathcal G_{\alpha}(\RR_{+})$.

We will  use the identity 
\begin{eqnarray} \label{eq:translacija}
e^{-itJ_{\rho,c}}  & = & 
\mathcal{J}^{-1}\circ e^{-it\rho E_{\rho,c}}\circ \mathcal {J} \nonumber  \\
 & = &  e^{-itc}\cdot(\mathcal{J}^{-1}\circ e^{-itE_{\rho,0}}\circ \mathcal {J}) 
  \nonumber  \\
  & = & e^{-itc}\cdot (\mathcal{J}^{-1}\circ e^{-it\rho E}\circ \mathcal {J}), \quad 
  \rho>0,t\in\RR\setminus\{0\}.
\end{eqnarray}

From  \eqref{eq:J-Phi-H-E}, \eqref{eq:frac1}, \eqref{eq:translacija} and the definitions of fractional Hankel transform, fractional powers of the Hankel--Clifford transform, and the fractional Fourier transform, we have  the following relation:
\begin{eqnarray}  \label{eq:frac3}
  ({\mathcal F}_{\rho} \modu{g})
  (x_1,x_2)
 & = & e^{i\frac{\pi\cdot \rho }{2}}\cdot \left( e^{-\frac{i\pi}{4}H_{\rho,0}}\modu{g}\right)
 (x_1,x_2)
 \nonumber \\
 & = & e^{i\frac{\rho\pi }{2}}\cdot\left(e^{-i\frac{\pi}{4} E_{4\rho,2\rho}} (\mathcal{J} \circ \Phi^{-1} \modu{g} )\right) (x_1^2+x_2^2)
 \nonumber \\
  & = & 
  \left( \mathscr{H}_{-\rho\pi}  (\Phi^{-1} \modu{g} ) \right)
 (\sqrt{x_1^2+x_2^2})
 \nonumber \\
 & = &
  \left( e^{-i\rho\pi E} 
 (\mathcal{J} \circ \Phi^{-1} \modu{g} ) \right)
 (x_1^2+x_2^2)
 \nonumber \\
 & = & 
  \left(\mathcal{I}_{e^{-i\rho\pi },0}
 (\mathcal{J} \circ \Phi^{-1} \modu{g} ) \right)
 (x_1^2+x_2^2), 
\end{eqnarray}
where $\modu{g}\in \mathcal H_{\frac{\alpha}{2},radial}(\RR^2)$, $\alpha>0$.

 By choosing $\rho = 1$ in \eqref{eq:frac3} we recover the well-known 
connection 
$$
 ({\mathcal F}_{1} \modu{g})
  (x_1,x_2) =  \left( \mathscr{H}_{\pi}  (\Phi^{-1} \modu{g} ) \right)
 (\sqrt{x_1^2+x_2^2}), \qquad x_1, x_2 \in \RR,
$$
between the Fourier transform of a function in $\RR^2$  and the Hankel transform of its radial part, see  \cite[p.~336]{Bracewell}.

The action of $\Phi^{-1}$, $\mathcal J$, the fractional Hankel transform and the 
fractional power of Hankel-Clifford transform is given in Figure \ref{Figure_1}.

\begin{figure}[h]
\[
\begin{tikzcd}[column sep=huge, row sep=huge]
\modu{g}
\arrow[r, "\Phi^{-1}"]
\arrow[d, "\mathcal J"']
& \Phi^{-1}\modu{g}
\arrow[r, "\mathscr{H}_{-\rho\pi}"]
\arrow[d, "\mathcal J"]
& \mathscr{H}_{-\rho\pi}(\Phi^{-1}\modu{g})
\arrow[r, "="]
& \mathcal F_{\rho}\modu{g}
\\
\mathcal J \modu{g}
\arrow[r, "\Phi^{-1}"']
& \mathcal (\mathcal J\circ \Phi^{-1})\modu{g}
\arrow[r,"\mathcal{I}_{e^{it},0}" ]
&  \mathcal{I}_{e^{it},0}\left((\mathcal J\circ \Phi^{-1})\modu{g}\right)
\arrow[u, "=", shift left]
\end{tikzcd}
\]
\caption{Relation between  fractional integral transforms - functions}\label{Figure_1}
\end{figure}

By the continuity of $e^{-i\rho\pi E} $, see Theorem \ref{cont. E},
it follows that all of the above mentioned transforms are continuous on their respective domains. 

We conclude the paper by deriving the dual counterpart of \eqref{eq:frac3}.

Let $\modu{g}\in \mathcal H'_{\frac{\alpha}{2},radial}(\RR^2)$. Then $g=\Phi'\modu{g}\in \mathcal H'_{\frac{\alpha}{2},even}(\RR)$, $f=\mathcal {(J'})^{-1}(g)\in\mathcal G'_{\alpha}(\RR) $,  and 
$$\left(e^{-i\frac{\pi}{4}H_{\rho,0 }}\right)\modu{g}=\left((\Phi')^{-1}\circ \mathcal{J}'\circ e^{-i\frac{\pi}{4}E_{4\rho,2\rho}}\circ \Phi'\circ (\mathcal{J'})^{-1}\right)\modu{g} . $$
From \eqref{eq:komutativnost-dual} we have $$\left(\Phi'\circ (\mathcal{J'})^{-1}\right)\modu{g}=\left((\mathcal{J'})^{-1}\circ\Phi' \right)\modu{g}=f$$ and therefore 
 $${\mathcal F}_{\rho} \modu{g}=  e^{-\frac{i\pi}{4}H_{\rho,0}}\modu{g}=\left((\Phi')^{-1}\circ \mathcal{J}'\circ \mathcal{I}_{e^{-i\rho\pi },0}\right)f.$$
 From \eqref{eq:fracdual1} we have that $\mathscr{H}_{-\rho\pi}g=\left(\mathcal{J'}\circ  \mathcal{I}_{e^{-\rho\pi i},0}\right)(f)$ 
 and, by  using \eqref{eq:komutativnost-dual}  once again,
 we obtain
\begin{equation*}
    {\mathcal F}_{\rho} \modu{g}=\left((\Phi')^{-1}\circ \mathcal{J}'\circ \mathcal{I}_{e^{-i\rho\pi },0}\right)f=
    \left((\Phi')^{-1}\circ\mathscr{H}_{-\rho\pi}\right)(g).
\end{equation*}

Therefore, Figure \ref{Figure_1} extends naturally to ultradistributions by replacing $\Phi^{-1}$ and $\mathcal{J}$ with $(\Phi')^{-1}$ and $\mathcal{J}'$, respectively, at each occurrence, see Figure \ref{Figure_2}.

\begin{figure}[h]
\[
\begin{tikzcd}[column sep=huge, row sep=huge]
\modu{g}
\arrow[r, "(\Phi')^{-1}"]
\arrow[d, "\mathcal J' "']
& (\Phi')^{-1}\modu{g}
\arrow[r, "\mathscr{H}_{-\rho\pi}"]
\arrow[d, "\mathcal J'"]
& \mathscr{H}_{-\rho\pi}((\Phi')^{-1}\modu{g})
\arrow[r, "="]
& \mathcal F_{\rho}\modu{g}
\\
\mathcal J' \modu{g}
\arrow[r, "(\Phi')^{-1}"']
& \mathcal (\mathcal J'\circ (\Phi')^{-1})\modu{g}
\arrow[r,"\mathcal{I}_{e^{it},0}" ]
&  \mathcal{I}_{e^{it},0}\left((\mathcal J'\circ (\Phi')^{-1})\modu{g}\right)
\arrow[u, "=", shift left]
\end{tikzcd}
\]
\caption{Relation between fractional integral transforms - ultradistributions}\label{Figure_2}
\end{figure}

\appendix 

\renewcommand{\thesection}{\Alph{section}}

\makeatletter
\def\@seccntformat#1{%
  \@ifundefined{#1@cntformat}
    {\csname the#1\endcsname.\hspace{0.5em}}
    {\csname #1@cntformat\endcsname}}
\newcommand\section@cntformat{%
  \appendixname\ \thesection.\hspace{0.5em}}
\makeatother

\section{Appendix} \label{App-A}

\renewcommand{\theequation}{A.\arabic{equation}}
\setcounter{equation}{0}

\subsection{Propagators as infinite sums of operators}

We first show that the notation in \eqref{eE^r} and \eqref{eq:general_propagator}
is in accordance to the usual definition of propagators in the form of a series. This is 
justified by the following result.

\begin{lem}\label{Lager} 
Let
$f\in\mathcal G_{\alpha}(\RR^d _+),$
$\alpha\in (0,1)$, and 
$E_{\rho,c}$ be given by  \eqref{eq:E_ro_ce}, $\rho\in\CC \setminus \{ 0\}$,  $c\in\CC$.
Then we have 
\begin{equation} \label{eq:semigroup}
e^{E_{\rho,c}}f (x)=\sum_{n=0}^\infty \frac{E_{\rho,c}^n }{n!} f (x),
\qquad 
x \in \overline{\RR^d _+}.    
\end{equation}

\end{lem}

\begin{proof}
Fix $x\in\overline{\RR^d_{+}},$  $m\in\NN_0^d,$ and $ n\in\NN_0$. Then
$$
\left(E^n_{\rho,c} f\right)(x)=\sum_{m\in\NN^d_0} a_m(f)\cdot  (\rho|m|+c)^{n}\cdot  l_m(x), 
$$ 
gives
$$\sum_{n=0}^{\infty} \frac{E_{\rho,c}^n f (x)}{n!}=\sum_{n=0}^{\infty}\left ( \sum_{m\in\NN_0^d} a_m(f)\cdot \frac{ \left( \rho|m|+c  \right)^{|n|} }{n!} l_m(x)       \right). $$
By Theorem \ref{glavnabeurling} 
we have $|a_m(f)|\leq C e^{-h |m|^{1/\alpha}}$ for some $C,h>0$,
so that
$$\left|a_m(f)\cdot \frac{ \left( \rho|m|+c  \right)^{n} }{n!} l_m(x)\right|\leq C e^{-h|m|^{1/\alpha}}\cdot \frac{ \left(|\rho| \cdot |m|+|c|\right)^{n} }{n!} \cdot C_1 |m|^{2d},$$
where we used the  estimate for Laguerre functions $|l_m(x)|\leq C_1 |m|^{2d}$, for some $C_1>0$, see \cite[Lemma $4.1$]{JPTV}.

If $\ds C_2=CC_1\sup_{m\in\NN_0^d} \left(e^{-\frac{h}{2}|m|^{1/\alpha}}\cdot |m|^{2d}\right)  $ we obtain
$$\left|a_m(f)\cdot \frac{ \left( \rho|m|+c  \right)^{n} }{n!} l_m(x)\right|\leq C_2 e^{-\frac{h}{2}|m|^{1/\alpha}} \cdot \frac{\left( |\rho|\cdot |m|+|c|\right)^{n}     }{(n!)^{1-\tau}}\cdot \frac{1}{(n!)^{\tau}}$$
for every $\tau\in (0,1)$. In the sequel we  choose $\tau>0 $ such that 
$\alpha + \tau <1$. 
From the inequality 
\begin{equation}\label{GS2}
C_1 e^{A s^{1/\alpha}}\leq \sup_{n\in\mathbb N_0}\frac{s^n}{n!^\alpha}\leq C_2 e^{B s^{1/\alpha}}, \qquad
 s\geq 0,
\end{equation}
see \cite[(9)]{JPTV},  it follows that 
there exists $C_3, h_1>0$ such that 
$$\frac{\left( |\rho|\cdot |m|+|c|\right)^{n}     }{(n!)^{1-\tau}}\leq C_3e^{\left( |\rho|\cdot |m|+|c|\right)^{1/(1-\tau)}  }.$$
Note that
$$
C_4=C_3 \cdot \sup_{m\in\NN_0^d}\left( e^{\left( |\rho|\cdot |m|+|c|\right)^{1/(1-\tau)} - \frac{h}{4}|m|^{1/\alpha}   } \right) < \infty,
$$
since $0 < \alpha < 1-\tau$. Finally, we have 
$$
\left|a_m(f)\cdot \frac{ \left( \rho|m|+c  \right)^{n} }{n!} l_m(x)\right|\leq C_4 e^{-\frac{h}{4}|m|^{1/\alpha}}\cdot \frac{1}{|n!|^{\tau}},
$$
and
$$
\sum_{m\in\NN_0^d,n\in\NN_0}e^{-\frac{h}{4}|m|^{1/\alpha}}\cdot \frac{1}{(n!)^{\tau}}<\infty,
$$
which proves the absolute convergence of the series
$$\sum_{m\in\NN_0^d,n\in\NN_0}a_m(f)\cdot \frac{ \left( \rho|m|+c  \right)^{n} }{n!} l_m(x).$$
This allows us to swap the order of summation, so that
\begin{eqnarray*}
 \sum_{n=0}^{\infty} \frac{E_{\rho,c}^n f (x)}{n!}
& = &
\sum_{m\in\NN_0^d}\left( \sum_{n=0}^{\infty}\left(a_m(f)\cdot l_m(x)\cdot  \frac{ \left( \rho|m|+c  \right)^{n} }{n!} \right) \right)
\\
 & = &
\sum_{m\in\NN_0^d} a_m(f)l_m(x)\cdot\left( \sum_{n=0}^{\infty} \frac{ \left( \rho|m|+c  \right)^{n} }{n!} \right)
\\
& = &\sum_{m\in\NN_0^d} a_m(f)\cdot e^{\rho|m|+c}\cdot l_m(x)=e^{\rho E+c}  f (x)
=e^{E_{\rho,c}}f (x),\end{eqnarray*}
which finishes the proof.
\end{proof}

The proof also implies the equality \eqref{eq:semigroup} holds uniformly with respect to  $x\in \overline{\RR^d_{+}}$.

The same technique can be used to prove a stronger statement, that is the series in \eqref{eq:semigroup} converges uniformly on bounded sets in  $\mathcal G_{\alpha}(\RR^d _+)$  
( in $\mathcal G_{0,\alpha}(\RR^d _+)$).
This implies that  $e^{-itE_{\rho,c}}\rightarrow I$ as  $t\to 0$,
in the strong operator topology on  
$(L(\mathcal G_{\alpha}(\RR^d _+))$ (on  $(L(\mathcal G_{0,\alpha}(\RR^d _+))$).

The following lemma shows that the condition $0< \alpha < 1$ 
is sharp when considering convergence in the $L^2$-norm.

\begin{lem}
Let $\alpha\geq 1$. Then there exists  $f\in \mathcal G_{\alpha}(\RR _+)$.
such that the series  $\ds \sum_{n=0}^\infty \frac{E_{\rho,c}^n f}{n!}$ 
does not converge in  the  $L^2$-norm . 
\end{lem}

\begin{proof}
It is enough to consider the case $\rho=1, c=0$ i.e. $E_{\rho,c}=E$.
From  \cite[(9)]{JPTV}, it follows that 
$$
\sup_{n\in\NN_0} \frac{k^n}{n!}\geq Ce^{kA}
\qquad \text{for some} \qquad C,A>0.
$$
If
$\ds f=\sum_{k=0}^{\infty} e^{-k}l_k\in \mathcal G_1(\RR _+)$
we obtain
$$\left\|\frac{E^n f}{n!}\right\|_{L^2}=\left(\sum_{k=0}^{\infty} e^{-2k}\cdot \frac{k^{2n}}{{n!}^2}\right)^{1/2}\geq \sup_{k\in\NN_0}\left( e^{-k}\frac{k^n}{n!}\right)=\frac{\sup_{k\in\NN_0} \left(e^{-k} k^n\right)}{n!}.$$
The function $x\mapsto e^{-x}x^n$ attains its maximum value when $x=n$, 
and by Stirling's formula we have 
$\ds n!\leq \sqrt{2\pi n}\cdot n^n e^{-n} e^{\frac{1}{12n}}$. 
Thus, we obtain
$$
\left\|\frac{E^n f}{n!}\right\|_{L^2}\geq \frac{e^{-n} n^n}{n!}\geq e^{-\frac{1}{12n}}\cdot \frac{1}{\sqrt{2\pi n}}\geq e^{-\frac{1}{12}}\frac{1}{\sqrt{2\pi n}}. 
$$
Since the series $\ds \sum_{n=1}^{\infty} \frac{1}{\sqrt{n}}$ is divergent, we conclude that $\ds \sum_{n=0}^{\infty}\frac{E^n f}{n!}$ does not converge either. 
\end{proof}

We conclude this subsection with the global counterpart of Lemma \ref{Lager} 
and omit the proof, since it is similar to that of Lemma \ref{Lager}.

\begin{lem}\label{Ermit}
Let there be given $\rho,c\in\CC$, and $f\in\mathcal H_{\alpha}(\RR^d)$, $ \alpha\in (0,\frac{1}{2})$. Then we have 
$$e^{H_{\rho,c}}f(x)=\sum_{n=0}^{\infty} \frac{H_{\rho,c}^n f(x)}{n!}, \qquad x \in \RR^d.$$
\end{lem}

\subsection{Radially invariant ultradistributions}

We end the paper by  studying the notion of radial invariance in the context of ultradistributions. Although the definition can be extended to much broader classes of ultradistributions, for simplicity we restrict our attention to the duals of $\mathcal A \in \{ \mathcal H_{\alpha/2}(\RR^d ), 
\mathcal  H_{0,\alpha/2}(\RR^d ) \}$,  $\alpha>0$.

\begin{defn} \label{def:radial-distr}
An  ultradistribution $f\in \mathcal {\mathcal A}' $ is called radially symmetric (or radial for short) if for every test function $\mathcal \varphi\in \mathcal A$ and every  $A\in O(d)$ we have 
$$\langle f\circ A,\varphi\rangle=\langle f,\varphi(A^{-1}\cdot)\rangle=\langle f,\varphi\rangle. $$
By $\mathcal A' _{radial}$ we denote the space of all radially symmetric  ultradistributions in 
$ \mathcal {\mathcal A}' $.
\end{defn}

It is easy to check that $\mathcal A' _{radial}$ is a closed subspace 
in $\mathcal A'$. 


As in \cite{Chung}, to every test function $\varphi \in \mathcal A$  we associate
its spherical average $\varphi_S(x)$, 
$$
\varphi_S (x)=\int_{O(d)} \varphi(Ax)dA=\int_{S^{d-1}}\varphi(|x|\omega)d\omega ,
\qquad x \in \RR^d
$$
where $dA$ is normalized Haar measure on $O(d)$. It is easy to see that this function is radial.

\begin{lem}\label{proj}
The mapping $P:\mathcal H_{\alpha/2}(\RR^d)\rightarrow {\mathcal H}_{\alpha/2, radial }(\RR^d)$ ($ P:\mathcal H_{0,\alpha/2}(\RR^d)\rightarrow {\mathcal H}_{0,\alpha/2, radial }(\RR^d)) $ defined by 
$  P(\varphi)=\varphi_S$, $ \varphi \in \mathcal A$,
is a continuous surjection and $P^2=P$.
\end{lem}

\begin{proof} We give the proof for $\mathcal H_{\alpha/2}(\RR^d)$, $\alpha>0$, 
since the proof for  $\varphi\in \mathcal H_{0,\alpha/2}(\RR^d)$ is similar.
Let $\varphi\in \mathcal H_{\alpha/2}(\RR^d)$, $\alpha>0$, i.e.
$\|H^N \varphi\|_{L^{\infty}(\RR^d)}\leq Ch^N N!^{\alpha}$ for some $C,h>0$ (see \eqref{eq:Pilipovic}). 
By
$(H\varphi(A\cdot))(x)=(H\varphi)(Ax)$, $A\in O(d)$,  we get
$$
(H^N \varphi_S)(x)=H^N\left(\int_{O(d)} \varphi(Ax)dA\right)=\int_{O(d)} (H^N \varphi)(Ax)dA,
\qquad N\in\NN_0,
$$
wherefrom
$$\|H^N \varphi_S\|_{L^{\infty}(\RR^d)}\leq C\cdot h^N N!^{\alpha},$$
which proves the claim. Since $P\varphi$ is radial, $P(P\varphi)= P\varphi$.
\end{proof}


Let $\mathcal A_{radial}$ be either $\mathcal H_{\alpha/2, radial}(\RR^d)$ or $\mathcal H_{0,\alpha/2, radial}(\RR^d)$. Recall,  $\mathcal A_{radial}$ is closed in $\mathcal A$.
Since $\mathcal H_{\alpha/2}(\RR^d)$ is $(FN)-$space, and therefore $(FS)-$space, and 
$\mathcal H_{0,\alpha/2}(\RR^d)$ is $(DFN)-$space, and therefore $(DFS)-$space, 
by \cite[Theorem A 6.5, p. 255]{morimoto} 
it follows that  
$(\mathcal A_{radial})'$, the dual space of $\mathcal A_{radial}$
is topologically isomorphic to
$\mathcal A'/{\mathcal A_{radial}^{\circ}}$, where 
$$\mathcal A_{radial}^{\circ}=\{f\in \mathcal A': \langle f,\varphi\rangle=0,\quad\mbox{for every}\quad \varphi\in\mathcal A_{radial}\}.$$  

Using this fact we are able to prove the following claim of independent interest.

\begin{prop}\label{dualradial}
The space of radial ultradistributions $\mathcal A'_{radial}$
given in Definition \ref{def:radial-distr}
is topologically isomorphic to $ (\mathcal A_{radial})'$,
the dual space of $\mathcal A_{radial}$.
\end{prop}

\begin{proof}
The idea of the proof is to construct a continuous bijection between $\mathcal A'_{radial}$ and $ (\mathcal A_{radial})'$,  and then apply the open mapping theorem, \cite[Corollary 1, p. 164]{Sch}.
We will use the above-mentioned fact that $ (\mathcal A_{radial})'$ is topologically isomorphic to $\mathcal A'/{\mathcal A_{radial}^{\circ}}$. 

Let $\Psi$ be the linear mapping 
$\Psi:\mathcal A'_{radial}\rightarrow \mathcal A'/{\mathcal A_{radial}^{\circ}}$ defined by 
$$
\Psi(f)=f+\mathcal A_{radial}^{\circ},  \qquad
f\in \mathcal A'_{radial}. 
$$

Assume that  $f+\mathcal A_{radial}^{\circ}=\mathcal A_{radial}^{\circ}$ for some $f\in  \mathcal  A'_{radial} $, i.e. $\langle f,\varphi\rangle=0$ for every $\varphi \in \mathcal A_{radial}$. 

Suppose now  that $ \varphi \in \mathcal A$. By using 
$$
0=\langle f,\varphi_S\rangle
= \int_{O(d)} \langle f,\varphi (Ax) \rangle dA
= \int_{O(d)} \langle f,\varphi (x) \rangle dA
=\langle f,\varphi\rangle, 
$$
it follows that $f=0$. 
Therefore, the mapping $\Psi$  is injective. 

In order to prove surjectivity, take the  arbitrary coset $f+\mathcal A_{radial}^{\circ}\in  \mathcal A'/{\mathcal A_{radial}^{\circ}}$. We define 
the functional $f_1$ that acts on $\varphi\in\mathcal A$ in the following way $\langle f_1,\varphi\rangle=\langle f, \varphi_S\rangle$. From the definition of $f_1$ and Lemma \ref{proj} it follows that $f_1\in\mathcal A'_{radial}$ and that $\langle f-f_1,\varphi\rangle=0$ for every $\varphi\in \mathcal A_{radial}$ i.e. $f+\mathcal A_{radial}^{\circ}=f_1+\mathcal A_{radial}^{\circ}$. Thus, $\Psi(f_1)=f+\mathcal A_{radial}^{\circ}$, which shows that $\Psi$ is surjective. 

Since $\Psi$ decomposes as $\mathcal A'_{radial}\rightarrow \mathcal A'\rightarrow \mathcal A'/\mathcal A_{radial}^{\circ}$ where the first mapping is canonical inclusion and the second is the canonical projection from $\mathcal A'$ onto its quotient space, it is continuous. 

Now, we can apply the open mapping theorem since
$\mathcal A'_{radial}$ is a Pt\'ak space and $\mathcal A'/\mathcal A_{radial}^{\circ}$ is barelled. It follows that $\Psi$ is a topological isomorphism.
\end{proof}

To establish the dual counterpart of Theorem \ref{Phi}, we use the transpose
$\Phi'$ of the operator $\Phi$ given in Proposition \ref{Phi}. 
Namely, for  $f\in {\mathcal H'_{\alpha/2, radial}(\RR^d)}$ ($f\in {\mathcal H'_{0,\alpha/2, radial}(\RR^d)}$) we denote
$\modu{f}=\Phi'(f)\in \mathcal H'_{\alpha/2,even}(\RR)$ ($\mathcal H'_{0,\alpha/2,even}(\RR)$)  i.e. 
\begin{equation}\label{phi'}
\langle \modu{f},\psi\rangle=
\langle \Phi'(f),\psi\rangle = 
\langle f, \Phi (\psi) \rangle =
\langle f, \modu{\psi}\rangle,
\end{equation}
$ \psi\in {\mathcal H}_{\alpha/2,even}(\RR) $ 
$(\psi\in {\mathcal H}_{0,\alpha/2,even}(\RR))$.
From  Propositions \ref{Phi} and  \ref{dualradial} 
we derive the following conclusion.

\begin{prop}\label{prop:phi'}
The mapping $\Phi'$  given by (\ref{phi'}) defines topological isomorphisms:
$$\Phi':{\mathcal H'}_{\alpha/2,radial}(\RR^d)\rightarrow {\mathcal H'}_{\alpha/2,radial}(\RR)$$ and 
$$\Phi':{\mathcal H'}_{0,\alpha/2,radial}(\RR^d)\rightarrow {\mathcal H'}_{0,\alpha/2,radial}(\RR).
$$
\end{prop}

\section*{Acknowledgement}

This work was supported by the
Science Fund of the Republic of Serbia,
{\#}GRANT No. 2727,
Global and local analysis of operators and
distributions - GOALS. 
S. Jak\v{s}i\'{c} also  acknowledges the financial support
of the Ministry of Science, Technological Development and Innovation of
the Republic of Serbia (Grant No. 451-03-34/2026-03/ 200169).
N. Teofanov also
gratefully acknowledges
the financial support
of the Ministry of Science, Technological 
Development and Innovation
of the Republic of Serbia
(Grants No. 451-03-33/2026-03/ 200125 \&
451-03-34/2026-03/ 200125), and partial support from "Quantum Austria" 
grant of the Austrian Research Foundation FWF (Grant number PAT 2056623) 
led by Maurice de Gosson (Vienna).

\end{document}